\documentclass[12pt,a4paper]{amsart}

\usepackage{amsmath}
\usepackage{amssymb}
\usepackage{amsfonts}
\usepackage{amsthm}
\usepackage[all]{xy}
\usepackage{tikz}
\usepackage{hyperref}
\usepackage[shortlabels]{enumitem}
\usepackage{multicol}
\usepackage[utf8]{inputenc}
\usepackage[polish,english]{babel}

\theoremstyle{plain}
\newtheorem*{atw*}{Theorem}
\newtheorem{atw}{Theorem}[section]
\newtheorem*{alemat*}{Lemma}
\newtheorem{alemat}[atw]{Lemma}
\newtheorem{pro}[atw]{Proposition}
\newtheorem*{pro*}{Proposition}
\newtheorem{cor}[atw]{Corollary}
\newtheorem*{cor*}{Corollary}

\newtheorem*{fact*}{Fact}
\theoremstyle{definition}
\newtheorem{adf}[atw]{Definition}
\newtheorem*{adf*}{Definition}
\theoremstyle{remark}
\newtheorem{rem}[atw]{Remark}
\newtheorem*{rem*}{Remark}
\newtheorem{ex}[atw]{Example}
\newtheorem*{ex*}{Example}

\newtheorem*{exer*}{Exercise}

\newcommand{\Z}{\mathbb{Z}}
\newcommand{\N}{\mathbb{N}}
\newcommand{\Q}{\mathbb{Q}}
\newcommand{\R}{\mathbb{R}}
\newcommand{\C}{\mathbb{C}}

\newcommand{\T}{\mathbb{T}}
\newcommand{\PP}{\mathbb{P}}

\newcommand{\ee}{{\bf e}}
\newcommand{\vv}{{\bf v}}

\newcommand{\ttt}{{\bf t}}

\newcommand{\xto}[1]{{\xrightarrow{#1}}}

\newcommand{\onto}{\twoheadrightarrow}

\newcommand{\Hom}{\operatorname{Hom}}
\newcommand{\codim}{\operatorname{codim}}
\newcommand{\coker}{\operatorname{coker}}
\newcommand{\lin}{\operatorname{span}}
\newcommand{\rank}{\operatorname{rank}}
\newcommand{\diag}{\operatorname{diag}}
\newcommand{\supp}{\operatorname{supp}}

\title[Comparison of motivic Chern classes and stable envelopes]{Comparison of motivic Chern classes and stable envelopes for cotangent bundles}
\author{Jakub Koncki}
\address{Institute of Mathematics, University of Warsaw, Poland}
\email{j.koncki@mimuw.edu.pl}

\def\red#1{#1}
\def\old#1{}

\begin{document}
	
\begin{abstract}
	We consider \red{a} complex smooth projective variety equipped with an action of \red{an} algebraic torus with \red{a} finite number
	of fixed points. We compare the motivic Chern classes of Białynicki-Birula cells \red{with} the $K$-theoretic stable envelopes of cotangent bundle. We prove that under certain geometric assumptions satisfied \red{e.g.} by homogenous spaces these two notions coincide up to normalization.
\end{abstract}
\maketitle		
\section{Introduction}
	\red{A} torus action on a smooth quasiprojective complex variety induces many cohomological structures. The Białynicki-Birula decomposition and the localization theorems are the \red{most} widely known examples. In this paper we aim to compare two \red{families of} $K$-theoretic characteristic classes induced by the Białynicki-Birula decomposition: the equivariant motivic Chern classes of Białynicki-Birula cells and the stable envelopes of cotangent \red{bundles}. For simplicity\red{,} we 
	\old{use shortcuts} \red{write} BB-decomposition and BB-cells for Białynicki-Birula decomposition and Białynicki-Birula cells\red{,} respectively.
	
	 Stable envelopes are characteristic classes defined for symplectic varieties equipped with torus \red{actions}. They are important objects in \old{the} modern geometric representation theory (cf. \cite{Op} for a survey).
	 Stable envelopes occur in three versions: cohomological \cite{OM}, $K$-theoretic \cite{OS,O2} and elliptic \cite{OA}. In this paper we focus on the $K$-theoretic ones. They depend on a choice of \red{a} linearisable line bundle called slope. Their axioms define \red{a} unique class
	 for general enough slope, yet \red{the} existence of \red{elements}
	 satisfying \red{the} axioms is still unknown in many cases.
	 
	 The motivic Chern class is an offshoot of the program of generalising characteristic classes of tangent \red{bundles} 
	 to \red{the} singular case.
	 It began with the construction of the Chern-Schwartz-MacPherson class in \cite{CSM} and was widely developed
	 (\red{see} e.g.	\cite{Oh,BSY,CMOSSY}
	 and \cite{SYp} for a survey). The common point of many of \old{defined} characteristic classes \red{that have been defined is} additivity properties with respect to \red{the} decomposition of a variety \red{as a union of} closed and open \red{subvarieties}. For example the non-equivariant motivic Chern class $mC_\red{y}$ (cf. \cite{BSY}) assigns to every map of varieties $\red{f:}X \to M$ a polynomial over \red{the} $K$-theory of coherent sheaves of $M$: an element \hbox{$\red{mC_y(X\xto{f}M)}\in G(M)[y]$.} Its additivity property states that:
	 $$mC_\red{y}(X\xto{f} M)=mC_\red{y}(Z\xto{f_{|Z}} M)+mC_\red{y}(X\setminus Z\xto{\red{f_{|X\setminus Z}}} M)\,,$$
	 for every closed subvariety $Z \subset X$. Similar properties are satisfied by the Chern-Schwartz-MacPherson class and the Hirzebruch class.
	 
	 Lately\red{,} equivariant versions of many such classes \red{have been} defined (e.g. \cite{Oh2, WeHir,FRW,AMSS}).
	 For an algebraic torus $\T \simeq \C^r$ the $\T$-equivariant motivic Chern class (cf. \cite{FRW,AMSS}) assigns to every $\T$-equivariant map of varieties $f: X \to M$ a polynomial over \red{the} $K$-theory of $\T$-equivariant coherent sheaves of $M$: an element $\red{mC^\T_y(X\xto{f}M)}\in G^\T(M)[y]$. It is  uniquely defined by \red{the following} three properties (after \cite{FRW}, section~2.3):
	 \begin{description}
	 	\item[1. Additivity] If \red{a $\T$-variety $X$ decomposes as a union of closed and open invariant subvarieties} $X=Y\sqcup U$, then $$mC_{\red{y}}^\T(X\xto{\red{f}} M)=mC_{\red{y}}^\T(Y\xto{\red{f_{|Y}}} M)+mC_{\red{y}}^\T(U\xto{\red{f_{|U}}} M)\,.$$
	 	
	 	\item[2. Functoriality] For \red{an equivariant} proper map $f:M\to M'$ we have $$mC_{\red{y}}^\T(X\stackrel{f\circ g}\to M')=f_*mC_{\red{y}}^\T(X\stackrel{g}\to M)\,.$$
	 	
	 	\item[3. Normalization] For a smooth \red{$\T$-}variety $M$ we have $$mC_{\red{y}}^\T(id_M)=\lambda_y(T^*M):=\sum_{i=0}^{\rank T^*M}[\Lambda^iT^*M]y^i \,.$$
	 	
	 \end{description}
 	In many cases\red{,} one can directly compute this class using the Lefschetz-Riemann-Roch theorem (cf. \cite{ChGi} theorem 5.11.7) and \red{the} above properties. For examples of \red{computations}
 	 see \cite{Feh,FRW,Kon}. In the paper \cite{FRW} (see also \cite{FRWp}) it was found that for $G$-equivariant varieties with \red{a} finite number
 	 of orbits\red{,} the motivic Chern classes $mC_{\red{y}}^G$ of $G$-orbits satisfy axioms similar to those of the stable envelopes.
 	
 	There is one more family of characteristic classes, called the weight functions,
 	\red{closely} connected \red{to the} characteristic classes mentioned above.
 	Their relations with other characteristic classes  were widely studied e.g. \cite{RTV',FRW,RW,KRW,RTV}.
 	
 	In this paper we consider a smooth projective variety $M$ equipped with an action of a torus $A$. Suppose that the fixed point set $M^A$ is finite. Our main result states that under some geometric assumptions on $M$
 	the stable envelopes for \red{a} small enough anti-ample slope and the motivic Chern classes coincide up to normalization. Namely:
 	\begin{atw*}
 		Let $M$ be as above. Consider the cotangent variety $X=T^*M$ with the action of the torus $\T=\C^*\times A$ (where the \red{first factor} $\C^*$ acts on fibers by scalar multiplication). Choose any weight chamber $\mathfrak{C}$ of the torus $A$ and polarization $T^{1/2}:=TM$. Suppose that the variety $M$ satisfies the local product condition
 		(see definition \ref{df:prd}). For any anti-ample $A$-linearisable line bundle $s$ and \red{a} sufficiently big integer $n$\red{,} the \red{element}
 		$$
 		\frac{mC_{-y}^A(M^+_F \to M)}{y^{\dim M_F^+}} \in K^A(M)[y,y^{-1}] \simeq K^\T(M) \simeq K^\T(T^*M)
 		$$
 		\old{determine} \red{is equal to} the $K$-theoretic stable envelope $y^{-\frac{1}{2}\dim M_F^+}Stab^{\frac{s}{n}}_{\mathfrak{C},T^{1/2}}(\red{1_F})$.
 	\end{atw*}
 	 \red{A} similar comparison was done in \cite{FR,RV,AMSS0} for the Chern-Schwartz-MacPherson class in the cohomological setting. The above theorem is a generalisation of the previous results of \red{\cite{AMSS,FRW}}
 	 where analogous equality
 	 is proved for the flag varieties $G/B$. \red{\cite{AMSS}} approach is based on \red{the} study of the Hecke algebra action on $K$-theory of flag variety, whereas our strategy is \red{similar to \cite{FRW}}. First\red{,} we make \red{a change} \old{correction}
 	 in \red{the} definition
 	 of the stable envelope such that it coincides with\old{the} Okounkov's\old{one} for general enough slope and is unique for all slopes (section \ref{s:env} and appendix \red{\hyperref[s:Ok]{A}}). By direct check of axioms\red{,} we prove that the equality from the theorem holds for the trivial slope (section \ref{s:mC}). \red{Then,} we check that the stable envelopes for \red{the} trivial slope and \red{a} small anti-ample slope coincide (section \ref{s:slope}).
 	 \red{Finally}, in \old{the} appendix \red{\hyperref[s:G/P]{B}} we check that \red{the} homogenous varieties $G/P$ satisfy \red{the} local product condition
 	 mentioned in the theorem.
 	 
 	 Our main technical tool is \red{the} study of
 	limits of Laurent polynomials (of one or many variables).
 	The limit technique was investigated \old{both} for motivic Chern classes \cite{WeBB,FRW,Kon} as well as for stable envelopes \cite{SZZ,O2}.
 	We use it mainly to prove various containments of Newton polytopes.
 	
 	Homogenous varieties $G/P$ are our main examples of varieties satisfying the local product condition. The study of characteristic classes and stable envelopes of such varieties is \red{an} important theme
 	present in recent research (e.g. \cite{AM,SZZ2,RV,RSVZ}).
 	A priori the stable envelope is defined for symplectic varieties which admit a proper map to an affine variety. This condition is satisfied by the cotangent bundles to  flag varieties of any reductive group and all homogenous varieties of $GL_n$.
 	There is \red{a} weaker condition
 	(see section \ref{s:env} condition ($\star$)) which is sufficient to define stable envelope and holds for  cotangent bundle to any variety which \red{satisfies}
 	the local product condition.
 	
 	\subsection{Acknowledgements}
 	I would like to thank \old{to} Andrzej Weber for his guidance and support.
 	I am grateful to Agnieszka Bojanowska for her valuable remarks.
 	I thank anonymous referees for their helpful comments. 
 	The author was supported by  the research project of the Polish National Research Center 2016/23/G/ST1/04282 (Beethoven 2, German-Polish joint project).

\section{Tools}
	This section gathers technical results, useful in the further parts of \red{this} paper. All considered varieties are assumed to be complex and quasiprojective.

\subsection{Equivariant K-theory}
	\red{Our main reference for the equivariant $K$-theory is \cite{ChGi}.}
	Let $X$ be a complex quasiprojective variety equipped with an action of a torus $\T$. We consider \red{the} equivariant $K$-theory of coherent sheaves $G^\T(X)$ and \red{the} equivariant $K$-theory of vector bundles $K^\T(X)$. For a smooth variety these two notions coincide. We use the lambda operations $\lambda_y:K^\T(X) \to K^\T(X)[y]$ defined by:
	$$ \lambda_y(E):=\sum_{i=0}^{\rank E}[\Lambda^iE]y^i \,. $$
	The operation $\lambda_{-1}:K^\T(X) \to K^\T(X)$ applied to the dual \red{bundle} is the $K$-theoretic Euler class. Namely
	$$eu(E)=\lambda_{-1}(E^*)\,. $$
	\old{For an immersion of a smooth subvariety $Y\subset X$ We denote by $eu(Y\subset X)$ the Euler class of normal bundle. Namely}
	\red{Let $Y\subset X$ be an immersion of a smooth \hbox{$\T$-invariant} locally closed subvariety. Its normal bundle is denoted by
	$$\nu(Y\subset X)=\coker (TY \to TX_{|Y}) \in K^\T(Y) \,.$$
	We denote by $eu(Y\subset X)$ the Euler class of normal bundle. Namely}
		$$eu(Y\subset X):=\lambda_{-1}(\nu^*(Y\subset X)) \in K^\T(Y) \,. $$
	
\begin{adf}
	Consider a \red{$\T$-variety} $X$.
	For an element $a \in G^\T(X)$ and a closed \red{invariant} subvariety $Y \subset X$ we say that $\supp(a) \subset Y$  if and only if $a$ \red{lies} in the image of \red{a} pushforward map
	$$G^\T(Y) \xto{i_*} G^\T(X).$$
	The short exact sequence (cf. \cite{ChGi} proposition 5.2.14)
	$$ G^\T(Y) \xto{i_*} G^\T(X) \to G^\T(X \setminus Y) \to 0 $$
	implies that \red{$\supp(a) \subset Y$} \old{this} is equivalent to $a_{|X \setminus Y}=0$.
\end{adf}
\begin{rem}
	Note that for an element $a \in K^\T(\red{X})$
	the support of $a$ \red{is not a} well defined subset
	of $X$. We \red{can} only define \red{the} notion
	$\supp(a) \subset Y$ for a closed subvariety $Y\subset X$. The fact that $\supp(a) \subset Y_1$ and $\supp(a) \subset Y_2$ \red{does not} imply that $\supp(a) \subset Y_1 \cap Y_2$.  
\end{rem}

\begin{pro}\label{cor:K}
	Consider a reducible $\T$-variety $X=X_1\cup X_2$. Denote the inclusions of the closed subvarieties $X_1$ and $X_2$ by $i$ and $j$\red{,} respectively. Then the pushforward map
	$$i_*+j_*: G^\T(X_1)\oplus G^\T(X_2) \to G^\T(X)$$
	is an epimorphism.
\end{pro}
\begin{proof}
	Denote by $U_1$ and $U_2$  the complements of the closed sets $X_1$ and $X_2$. Note that due to the exact sequence
	$$ G^\T(X_1) \xto{i_{*}} G^\T(X) \to G^\T(U_1) \to 0$$
	it is enough to prove that the composition
	$$\alpha: G^\T(X_2) \xto{j_{*}} G^\T(X) \to G^\T(U_1)$$
	is an epimorphism. Note that $U_1 \cap X_2=U_1$ and by pushforward pullback argument the map $\alpha$ is equal to the restriction to open subset
	$$G^\T(X_2)\to G^\T(U_1).$$
	Such restriction
	is \red{an} epimorphic map
	due to the exact sequence of \red{a} closed immersion.
\end{proof}
\red{Consider a $\C^*$-variety $F$ for which the action is trivial. Every equivariant vector bundle  $E\in Vect^{\C^*}(F)$ decomposes as a sum of $\C^*$-eigenspaces
$$E=\bigoplus\limits_{n\in \Z} E_n \,.$$
The sum $E^+=\oplus_{n>0} E_n$ is called the attracting (or positive) part of $E$ while the the sum $E^-=\oplus_{n<0} E_n$ is called the repelling (or negative) part.
The assignment of the positive (or negative) part induces a map of the equivariant K-theory. Namely}
\begin{pro}
	\label{lem:attr}
	Let $\sigma \subset \T$ be a one dimensional subtorus. Suppose that $F$ is a $\T$-variety \red{for which the action of $\sigma$ is trivial}. Then taking \red{the} attracting part with respect to the torus $\sigma$ \old{of a vector bundle} induces a well defined map
	$$K^\T(F) \to K^\T(F)\,.$$
	\red{An} analogous result holds for \red{the} repelling part. More generally taking direct summand corresponding to a chosen character of $\sigma$ induces such a map.
\end{pro}
\begin{proof}
	$\T$-equivariant maps of vector bundles preserve \red{the} weight decomposition with respect to the torus $\sigma$.
	Thus\red{, any} exact sequence of $\T$-vector bundles splits into \red{a} direct sum of sequences corresponding to characters of $\sigma$.
	It follows that taking the part corresponding to any subset of characters preserves exactness.
\end{proof}
\begin{rem}
	\red{Denote by $R(\sigma)$ the representation ring of the torus $\sigma$.} Proposition \ref{lem:attr} is a consequence of an isomorphism
	$$K^\T(F) \simeq K^{\T/\sigma}(F)\otimes R(\sigma) .$$
\end{rem}

\subsection{BB-decomposition}
 The BB-decomposition was introduced in \cite{B-B1} and further studied in \cite{B-B3} (see also \cite{CarBB} for a survey).
 We recall here its definition and fundamental properties. Consider a smooth $\sigma=\C^*$-variety $X$.
	\begin{adf} \label{def:BB}
		Let $F$ be a component of the fixed point set $X^\sigma$. The positive BB-cell of $F$ is the subset
		$$X_F^+=\{ x\in X\;|\; \lim_{t\to 0} t\cdot x\in F\} \,.$$
		Analogously the negative BB-cell of $F$ is the subset
		$$X_F^-=\{ x\in X\;|\; \lim_{t\to \infty} t\cdot x\in F\} \,.$$	
	\end{adf}
It follows from \cite{B-B1} that
	\begin{atw} \label{tw:BB}
		\begin{enumerate}
			\item The BB-cells are locally closed, smooth, algebraic subvarieties of~$X$. Moreover, we have the equality of vector bundles
			$$T(X_F^+)_{|F}=(TX_{|F})^+\oplus TF $$
			\item  There exists an algebraic morphism
			$$\lim_{t\to 0}: X_F^+\to F \,.$$
			\item Suppose that the variety $X$ is projective. Then there is a set decomposition (called BB-decomposition)
			$$X=\bigsqcup_{F\subset X^\sigma} X_F^+ \,.$$
			\item Suppose that the variety $X$ is projective. Then the morphism $\lim_{t\to 0}: X_F^+\to F$ is an affine bundle.
			\item Suppose that a bigger torus $\sigma\subset\T$ acts on $X$. Then the BB-cells (defined by the action of $\sigma$) are $\T$-equivariant subvarieties.
			\item The BB-decomposition induces a partial order on the fixed point set $X^\sigma$, defined by the transitive closure of the relation
			$$F_2\in\overline{X^+_{F_1}} \Rightarrow F_1>F_2 \,.$$
		\end{enumerate}
	\end{atw}
\begin{adf}[cf. \cite{OM} paragraph 3.2.1]
	Suppose that \red{$X$ is a smooth $\T$-variety.}
	Consider the space of cocharacters
	$$\mathfrak{t}:=\Hom(\C^*,\T)\otimes_\Z\R \,.$$
	For a fixed points component $F \subset X^\T$, denote \red{by $v_1^F,...,v_{\codim F}^F$} the torus weights appearing in the normal bundle.
	\red{A} weight chamber
	 is a connected component of the set
	$$\mathfrak{t} \setminus\bigcup_{F\subset \red{X^\T}, i \le \text{codim}F}\{v_i^F=0\}.$$
\end{adf}

Suppose that a torus $\T$ acts on a smooth variety $X$. For \red{a} one dimensional subtorus $\sigma \subset \T$ and a weight chamber $\mathfrak{C}$ we \red{write}
 $\sigma \in \mathfrak{C}$ when the cocharacter of $\sigma$ belongs to the chamber $\mathfrak{C}$. 

\begin{pro} \label{lem:st}
	\red{Let $X$ be a smooth $\T$-variety.} Consider one dimensional subtorus $\sigma \subset \T$ such that $\sigma \in \mathfrak{C}$ for some weight chamber $\mathfrak{C}$. Then the fixed \red{point} sets
	$X^\T$ and $X^\sigma$ are equal.
\end{pro}

\begin{pro}
	\label{lem:chamb}
		\red{Let $X$ be a smooth $\T$-variety.} Choose a weight chamber $\mathfrak{C}$.
		Consider one dimensional subtori \hbox{$\sigma_1,\sigma_2 \subset \T$} such that $\sigma_1,\sigma_2 \in \mathfrak{C}$. Then the tori $\sigma_1$ and $\sigma_2$ induce the same decomposition of \red{the} normal bundle to the fixed \red{point} set
		\red{into} the attracting and \red{the}
	repelling
	part. Moreover\red{,} the BB-decompositions with respect to these tori are equal.
\end{pro}
\begin{proof}
	The only nontrivial part is \red{the} equality of the BB-decompositions. It is \red{a} consequence
	of theorem 3.5 of \cite{HuBB}.
	Alternatively\red{,} thanks to the Sumihiro theorem \cite{Su} it is enough to prove the lemma for $X$ equal to the projective space. In this case\red{,} the proof is straightforward.
\end{proof}

\subsection{Symplectic varieties}

\begin{adf}[\cite{OS} section 2.1.2]
	\label{df:pol}
	Consider a smooth symplectic variety $(X,\omega)$ \red{equipped} with an action of a torus $\T$. Assume that the symplectic form $\omega$ is an eigenvector of $\T$, let $h$ be its character. \red{A} polarization
	is 
	an element $T^{1/2} \in  K^\T(X)$
	such that
	$$T^{1/2}\oplus \C_{-h}\otimes(T^{1/2})^*=TX \in  K^\T(X) \,.$$ 
\end{adf} 

\begin{rem}
	Note that according to the above definition\red{, a} polarization
	 is an element \red{of the} $K$-theory, thus it may be a virtual vector bundle.
\end{rem}

\subsection{Newton polytopes}

	Let $R$ be \red{a commutative ring with unit} and $\Lambda$ a lattice of finite rank. Consider a polynomial $f \in R[\Lambda]$. The Newton polytope $N(f) \subset \Lambda\otimes_\Z \R$ is a convex hull of lattice points corresponding to the nonzero coefficients of the polynomial $f$. We \red{recall} elementary properties of Newton polytopes:
	
\begin{pro} \label{lem:New}
	Let $R$ be \red{a commutative ring with unit}.
	 For any Laurent polynomials $f,g\in R[\Lambda]$.
	\begin{enumerate}[(a)]
		\item $N(fg) \subseteq N(f)+N(g)$.
		\item $N(fg) = N(f)+N(g)$ when the ring $R$ is a domain.
		\item $N(fg) =N(f)+N(g)$ when \red{the} coefficients of the class $f$ corresponding to the vertices of \red{the} polytope $N(f)$ \red{are not} zero divisors.
		\item $N(f+g) \subseteq conv(N(f),N(g))$.
		\item Let $\theta: R \to R'$ be a homomorphism of rings
		and $\theta': R[\Lambda] \to R'[\Lambda]$ its extension.
		Then $N(\theta'(f))\subseteq N(f)$.
	\end{enumerate}
\end{pro}

Consider a smooth variety $X$ equipped with an action of a torus $\T$. Choose a subtorus $A \subset \T$. For \red{a fixed point set}
component $F\subset X^\T$ and an element $a\in K^\T(X)$ we want to \red{define}
the Newton polytope
$$N^A(a_{|F})\subset \Hom(A,\C^*)\otimes_Z\R :=\mathfrak{a}^*. $$
It is possible due to the following proposition:

\begin{pro} \label{lem:N(a)}
	\red{Let $F$ be a smooth $\T$-variety.}
	 Let  $A\subset \T$ be a subtorus which acts trivially on $F$.
	 Any \red{splitting}
	 of the inclusion $A\subset \T$ induces isomorphism
	$$\alpha: K^\T(F) \simeq K^{\T/A}(F) \otimes_{\Z} R(A) \simeq K^{\T/A}(F)[\Hom(A,\C^*)] 
	\,,$$
	\red{where $R(A)$ denotes the representation ring of the torus $A$.}
	For an element $a \in K^\T(F)$ we consider the Newton polytope $N^A(a) \subset \mathfrak{a}^*$
	\red{defined by the polynomial $\alpha(a)$}.
	The isomorphism $\alpha$ depends on the choice of \red{splitting},
	 yet the Newton polytope is independent \red{of} it.
\end{pro}
\begin{proof}
	Consider two \red{splittings}
	$s_1,s_2: \T/A \to \T$. Denote by $\alpha_1,\alpha_2$ the induced isomorphisms
	$$K^{\T}(F) \to K^{\T/A}(F)[\Hom(A,\C^*)]\,. $$
	Note that the quotient $\frac{s_2}{s_1}$ induces a group homomorphism \hbox{$h:\T/A \to A$.} Consider an arbitrary class $E \in K^{\T/A}(F)$ and a character $\chi \in \Hom(A,\C^*)$. Direct calculation \red{provides}
	us with the formula 
	$$\alpha_2\circ\alpha_1^{-1}(Ee^{\chi})=\left(E\otimes\C_{\chi\circ h}\right)e^{\chi} \,.$$
	Thus\red{,} the Newton polytope is independent \red{of} the choice of
	\red{splitting}.
\end{proof}
\begin{rem} \label{rem:N}
	Consider the situation as in proposition \ref{lem:N(a)}. Let $E$ be a $\T$-vector bundle over $F$. Then the Euler class $eu(E)$ satisfies \red{the} assumption of proposition \ref{lem:New} (c). To see this use \red{the} weight decomposition of $E$ with respect to the torus $A$. Note that for a vector bundle $V$ with \red{an} action of $A$ given by a single character the Newton polytope $N^A(eu(V))$ is an interval. Moreover\red{,} the coefficients corresponding to the ends of this interval are invertible (they are equal to classes of the line bundles $1$ and $\det V^*$).
\end{rem}

At the end of this section we want to \red{mention the} behaviour of \old{the} polytope $N^A$ after restriction to one dimensional subtorus of $A$. Namely let $\sigma \subset A$ be a one dimensional subtorus. Denote by
$$|_\sigma:K^A(pt) \to K^{\sigma}(pt)$$
\red{the} induced map on the $K$-theory and by
$$\pi_\sigma:\Hom(A,\C^*) \otimes_\Z \R \to \Hom(\sigma,\C^*) \otimes_\Z \R$$
\red{the} induced map on the characters.

\begin{pro} \label{lem:Npi}
	For an element $a \in K^\T(F)$ there exists a finite union of hyperplanes $K$ in the vector space of cocharacters such that for all one dimensional subtori $\sigma$ whose cocharacter \red{does not} belong to $K$
	$$N^\sigma(a|_\sigma)=\pi_\sigma \left(N^A(a)\right) .$$
\end{pro}

\begin{pro} \label{lem:ver}
 \red{Let $M$ be a smooth $A$-variety.}
 Consider \red{a} one dimensional subtorus
 $\sigma \subset A$ such that $\sigma \in \mathfrak{C}$ for some weight chamber $\mathfrak{C}$. Let $F$ be a component of \red{the fixed point set $M^A$.}
  Then the point $0$ is a vertex of the polytope $N^A(eu(\nu_F^-))$. Moreover\red{,} the point $\pi_\sigma(0)$ is the minimal term of \red{the} line segment
 $\pi_\sigma\left(N^A(eu(\nu_F^-))\right)$.
\end{pro}

\section{Stable envelopes for isolated fixed points} \label{s:env}
In this section we recall the definition of the $K$-theoretic stable envelope introduced in \cite{OS,O2} (see also \cite{SZZ,AMSS,RTV} for
special cases) in the case of isolated fixed points. 
In \old{the} appendix \red{\hyperref[s:Ok]{A}} we give \red{a} rigorous proof that such classes are unique (proposition \ref{pro:uniq}). Using \old{the} Okounkov's definition it is true only for a general enough slope. We introduce \red{a weaker version of the axioms} \old{correction in the definition} to bypass this problem.

We use the following notations and \red{assumptions}
\begin{itemize}
	\item $\T \simeq \C^r$ is an algebraic torus.
	\item $(X,\omega)$ is a smooth symplectic $\T$-variety.
	\item $A \subset \T$ is a subtorus which preserves the symplectic form.
	\item We assume that the fixed point set $X^A$ is finite (\red{this} implies that $X^A=X^\T$).
	\item We assume that $\omega$ is an eigenvector of $\T$ \red{and} denote its character by \hbox{$h \in \Hom(\T,\C^*)$.} 

	\item $\mathfrak{C} \subset \mathfrak{a}$ is a weight chamber.
	\item For a fixed point $F \in X^A$\red{,} we denote by $X_F^+$ its positive BB-cell \old{depending (only) on} \red{according to} the chamber $\mathfrak{C}$ (cf. \red{definition \ref{def:BB} and} proposition \ref{lem:chamb}).
	\item We denote by $\ge$ the partial order on the fixed \red{point} set \red{$X^A$}
	 induced by the chamber~$\mathfrak{C}$ (cf. theorem \ref{tw:BB} (6)).
	\item For a fixed point $F \in X^A$\red{, let $$\nu_F:=\nu(F\subset X) \simeq TX_{|F} \,,$$
	be the normal bundle to the inclusion $F\subset X$.
	Let $\nu_F=\nu_F^+\oplus\nu_F^-$ be its decomposition}
	induced by the chamber $\mathfrak{C}$. (cf. proposition \ref{lem:chamb}).
	\item $T^{1/2} \in K^\T(X)$ is a polarization  (cf. definition \ref{df:pol}).
	\item For a fixed \red{point} \old{set
	component} $F\in X^A$\red{,} we denote by $T^{1/2}_{F,>0} \in K^\T(F)$ the attracting part 
	of
	$(T^{1/2})_{|F} \in K^\T(F)$ (cf. proposition \ref{lem:attr}). 
	\item $s \in Pic(X) \otimes \Q$ is a rational $A$-linearisable line bundle which we call slope.
	\red{\item For a fixed point $F\in X^A$, we denote by $1_F$ the multiplicative unit in the ring $K^\T(F)$ given by the class of the equivariant structure sheaf.}
\end{itemize}

Moreover\red{,} we assume that the variety $X$
satisfies the following condition:
\begin{enumerate}[$(\star)$]
	\item
	The set $\bigsqcup_{F \in X^A} X^+_{F}$ is closed.
\end{enumerate}
\begin{rem} 
	The  $(\star)$ condition \red{implies} that for any fixed point $F_0\in X^A$ the set $\bigsqcup_{F\in X^A, F\le F_0} X^+_{F}$ is closed.
\end{rem}
In \old{the} Okounkov's papers stronger condition
\red{on $X$ is assumed. Namely it is required that $X$ admits an}
 equivariant proper map to an affine variety cf. \cite{OM} paragraph 3.1.1.
Existence of such \red{a} map implies the condition $(\star)$ cf.  \cite{OM} lemma 3.2.7.
It turns out that the condition~$(\star)$ is sufficient to prove uniqueness of the stable envelope (cf. proposition~\ref{pro:uniq}).

\begin{adf} [cf. chapter 2 of \cite{OS}] \label{df:env}
	The stable envelope is a morphism of \hbox{$K^\T(pt)$-modules}
	$$Stab_{\mathfrak{C},T^{1/2}}^s: K^\T(X^A) \to K^\T(X)$$
	satisfying three properties:
	\begin{enumerate}[{\bf a)}]
		\item For any fixed point $F$
		$$\supp\left(Stab_{\mathfrak{C},T^{1/2}}^s(1_F)\right) \subset \bigsqcup_{F' \le F} X^+_{F'} \,.$$
		\item For any fixed point $F$
		$$Stab_{\mathfrak{C},T^{1/2}}^s(1_F)_{|F}=
		eu(\nu^-_F)\frac{(-1)^{\rank T^{1/2}_{F,>0}}}{\det T^{1/2}_{F,>0}} h^{\frac{1}{2}\rank T^{1/2}_{F,>0}} \,.$$
		\item Choose any $A$-linearisation of the slope $s$. For a pair of fixed points $F',F$ such that $F'<F$
		we demand \red{containment of Newton polytopes}
		$$N^A\left(Stab_{\mathfrak{C},T^{1/2}}^s(1_F)_{|F'}\right)+s_{|F}
		\subseteq
		\left(N^A(eu(\nu^-_{F'})) \setminus \{0\}\right)-\det T^{1/2}_{F',>0} +s_{|F'},$$
		where the Newton polytopes are defined as in \old{the}
		proposition \ref{lem:N(a)}.
		An addition
		of \red{the restriction of} a line bundle \red{means} translation by its character.
\end{enumerate}
\end{adf}
For a comparison of the above definition with the one given in \cite{OS,O2} see appendix~\red{\hyperref[s:Ok]{A}}.
\begin{rem} 
To define the element $h^{\frac{1}{2}\rank T^{1/2}_{F,>0}}$ in the axiom {\bf b)} one may need to pass to the double cover of \red{the} torus $\T$ (cf. paragraph 2.1.4 in \cite{OS}). To avoid this problem we consider \red{normalized version of the stable envelope (see the expression (\ref{wyr:stab}) below).} \old{the morphisms $h^{-\frac{1}{2}\rank T^{1/2}_{F,>0}}Stab_{\mathfrak{C},T^{1/2}}^s$.}
\end{rem}
\begin{rem} \label{rm:uni}
	\old{The} Okounkov's definition differs from the one given above in the axiom {\bf c)}. In the paper \cite{OS} weaker set containment is required 
		$$N^A\left(Stab_{\mathfrak{C},T^{1/2}}^s(\red{1_F})_{|F'}\right)+s_{|F}
	\subseteq
	N^A(eu(\nu^-_{F'})) -\det T^{1/2}_{F',>0} +s_{|F'}.$$
	It defines the stable envelope uniquely only for \red{a} general enough slope (cf. paragraph 2.1.8 in \cite{OS}, proposition 9.2.2 in \cite{O2} and example \ref{ex:uni}). With our \red{version of axioms} \old{correction} the stable envelope is unique for any choice of slope and coincides with \old{the} Okounkov's\old{one} for general enough slope. \old{(namely such that $s_{|F}-s_{|F'}$ is not an integral point)}
\end{rem}

	For simplicity\red{,} we omit the weight chamber and polarization in the notation. The stable envelope is determined by the set of elements
	\begin{align} \label{wyr:stab}
	Stab^\red{s}(F) :=h^{-\frac{1}{2}\rank T^{1/2}_{F,>0}}Stab(1_F), 
	\end{align}
	indexed by \red{the fixed point set $X^A$.}
	It leads to the following equivalent definition.
	\begin{adf} \label{def:ele}
		 The $K$-theoretic stable envelope is a set of elements $Stab^\red{s}(F)\in K^\T(X)$ indexed by \red{the fixed point set $X^A$}, such that
			\begin{enumerate}[{\bf a)}]
			\item For any fixed point $F$ 
			$$\supp(Stab^\red{s}(F)) \subset \bigsqcup_{F' \le F} X^+_{F'} \,.$$
			\item For any fixed point $F$
			$$Stab^\red{s}(F)_{|F}=
			eu(\nu^-_F)\frac{(-1)^{\rank T^{1/2}_{F,>0}}}{\det T^{1/2}_{F,>0}} \,.$$
			\item Choose any $A$-linearisation of the slope $s$. For a pair of fixed points  $F',F$ such that $F'<F$ we demand  \red{containment of Newton polytopes} 
			$$N^A\left(Stab^\red{s}(F)_{|F'}\right)+s_{|F}
			\subseteq
			\left(N^A(eu(\nu^-_{F'}))\setminus\{0\}\right)-\det T^{1/2}_{F',>0} +s_{|F'}.$$ 
		\end{enumerate}
	\end{adf}

\begin{pro} \label{pro:uniq}
	\old{With} \red{Under the} assumptions given at the beginning of this section the stable envelope (definitions \ref{df:env}, \ref{def:ele}) is unique. 
\end{pro}
	
The simplest example of a symplectic variety is \red{the} cotangent \red{bundle}
to a smooth variety.
It is natural to ask whether such a variety satisfies the assumptions needed to define $K$-theoretic stable envelope.
Consider a smooth \red{$A$-variety $M$}
with \red{a} finite number 
of fixed points. Consider the cotangent variety $X=T^*M$ with the action of the torus $\T=A \times \C^*$ such that $\C^*$ acts on the fibers by scalar multiplication.
	The fixed \red{point} set
	 of this action is finite. In fact\red{,} we have \red{equalities}
	$$X^{\T}=X^A=M^A.$$
	The variety $X$ is equipped with the canonical symplectic nondegenerate form $\omega$. This form is preserved by the torus $A$ and it is an eigenvector of the torus $\T$ with character corresponding to \red{the} projection on the second factor. Denote this character by $y$.
	The subbundle $TM \subset TX$ satisfies the polarization condition (see definition \ref{df:pol}). Choose a weight chamber $\mathfrak{C}$ of the torus $A$ and a one dimensional subtorus $\sigma \subset A$ such that $\sigma \in \mathfrak{C}$. The BB-cells of $\sigma$ in $X$ are the conormal bundles to the BB-cells of $\sigma$ in $M$.
	The condition $(\star)$ means that the set
	$$\bigcup\limits_{F \in M^A} \nu^*(M^+_F)$$
	is a closed subset of $T^*M$. This condition is not always satisfied.
	\begin{rem}
		Consider a smooth manifold $M$ decomposed as \red{a} disjoint union
		of smooth locally closed submanifolds
		$$M=\bigsqcup_i S_i .$$
		\red{The disjoint union}
		of the conormal bundles to strata is a closed subset of the cotangent bundle $T^*M$ if and only if the decomposition satisfies the Whitney condition (A) (cf.
		\cite{Whcond} exercise 2.2.4).
		 Thus\red{,} the $(\star)$ condition can be seen as \red{an} algebraic counterpart of the Whitney condition~(A).
	\end{rem}
	\begin{ex}
		Consider the projective space $\PP^2$ with the action of \red{the} diagonal torus of $GL_3(\C)$. Choose a general enough subtorus $\sigma$. Let $x$ be the middle fixed point in the Bruhat order. Then the cotangent bundle to \red{the} blow-up $T^*(Bl_x\PP^2)$ \red{does not} satisfy the $(\star)$ condition.
	\end{ex}
		\begin{ex}
			\begin{enumerate}
				\item The cotangent bundle to any  flag variety $G/B$ admits an equivariant projective morphism to an affine variety (cf. \cite{ChGi} section 3.2) 
				so it satisfies $(\star)$ condition.
				\item The cotangent bundle to any $GL_n$ homogenous variety is a Nakajima quiver variety (cf. \cite{Nak} section 7) so it admits an equivariant projective morphism to an affine variety (\cite{Nak2} section 3, \cite{GiNak} section 5.2).
			\end{enumerate}
	\end{ex}
	We introduce a stronger condition on the BB-cells of $M$ which \red{implies}
	the ($\star$) condition for $X=T^*M$ and is satisfied by the homogenous varieties (cf. theorem \ref{tw:prod}).
	\begin{adf}\label{df:prd}
		Consider a projective smooth variety $M$ \red{equipped} with an action of a torus $A$ with \red{a} chosen one dimensional subtorus $\sigma$. Suppose that the fixed \red{point} set
		$M^\sigma$ is finite.
		We say that $M$ satisfies the local product condition
		if for any fixed point $x \in M^\sigma$ there exist an $A$-equivariant Zariski open neighbourhood $U$ of $x$ and  $A$-variety $Z_x$ such that:
		\begin{enumerate}
			\item There \red{exists}
			 an $A$-equivariant isomorphism $$\theta: U \simeq (U\cap M_x^+)\times Z_x \,. $$
			\item For any fixed point $y \in M^\sigma$ there \red{exists} a subvariety $\red{Z'_{x,y}}\subset Z_x$ such that $\theta$ induces isomorphism:
			$$U\cap M_y^+ \simeq (U\cap M_x^+)\times \red{Z'_{x,y}} \,.$$
		\end{enumerate}	
	\end{adf}

	\begin{pro}
		Suppose that a projective smooth \red{$\C^*$}-variety $M$
		 satisfies the local product condition. Then the cotangent variety $T^*M$ satisfies the $(\star)$ condition.
	\end{pro}
	\begin{proof}
		It is enough to prove that for every fixed point $F_0 \in M^\sigma$ there is an inclusion
		$$\overline{\nu^*(M_{F_0}^+\subset M)}\subset \bigsqcup_{F\in M^\sigma}\nu^*(M_{F}^+\subset M) \,. $$
		It is equivalent to claim that for arbitrary fixed points $F_0,F$ 
		$$\overline{\nu^*(M_{F_0}^+\subset M)}\cap T^*M_{|M_F^+}\subset \nu^*(M_{F}^+\subset M) \,. $$
		Denote by $U$ the neighbourhood of $F$ from \red{the} definition of the local product condition. All of the subsets in the above formula are $\sigma$ equivariant. Thus\red{,} it is enough to prove that
		$$\overline{\nu^*(M_{F_0}^+\cap U\subset U)}\cap T^*U_{|M_F^+}\subset \nu^*(M_{F}^+\cap U\subset U) \,. $$
		The \red{local}
		product property implies existence of isomorphisms
		$$
		U\simeq M_F^+\times Z, \ \
		M_F^+ \simeq M_F^+\times \{pt\}, \ \
		M_{F_0}^+ \simeq M_F^+\times Z' \,,
		$$
		for some subvariety $Z'\subset Z$ and point $pt\in Z$. Denote by $E$ the subbundle $$M_F^+\times T^*Z \subset T^*U.$$
		Note that
		\begin{align*}
			&\nu^*(M_{F}^+\cap U\subset U)=E_{|M_{F}^+} \\
			&\nu^*(M_{F_0}^+\cap U\subset U)=M_F^+\times \nu^*(Z'\subset Z) \subset E_{|M^+_{F_0}}
		\end{align*}
		Thus
		\begin{multline*}
		\overline{\nu^*(M_{F_0}^+\cap U\subset U)}\cap T^*U_{|M_F^+}\subset \overline{E_{|M^+_{F_0}}} \cap T^*U_{|M_F^+} \subset
		\\
		\subset E \cap T^*U_{|M_F^+}=E_{|M_F^+}= \nu^*(M_{F}^+\cap U\subset U) \,.
		\end{multline*}
	\end{proof}
	 
	\red{\section{Motivic Chern class}
	The motivic Chern class is defined in \cite{BSY}. The equivariant version is due to \cite[section 4]{AMSS} and \cite[section 2]{FRW}. Here we recall the definition of the torus equivariant motivic Chern class. Consult \cite{AMSS,FRW} for a detailed account.}

		\red{\begin{adf}[after {\cite[section 2.3]{FRW}}]
			Let $A$ be an algebraic torus.
			The motivic Chern class assigns to every $A$-equivariant map of quasi-projective $A$-varieties $f:X \to M$ an element
			$$mC_y^A(f)=mC_y^A(X \xto{f} M) \in G^A(M)[y]$$
			such that the following properties are satisfied
			\begin{description}
				\item[1. Additivity] If a $A$-variety $X$ decomposes as a union of closed and open invariant subvarieties $X=Y\sqcup U$, then $$mC_y^A(X\xto{f} M)=mC_y^A(Y\xto{f_{|Y}} M)+mC_y^A(U\xto{f_{|U}} M)\,.$$
				\item[2. Functoriality] For an equivariant proper map $f:M\to M'$ we have $$mC_y^A(X\stackrel{f\circ g}\to M')=f_*mC_y^A(X\stackrel{g}\to M)\,.$$
				\item[3. Normalization] For a smooth $A$-variety $M$ we have $$mC_y^A(id_M)=\lambda_y(T^*M):=\sum_{i=0}^{\rank T^*M}[\Lambda^iT^*M]y^i \,.$$ 
			\end{description}
		The motivic Chern class is the unique assignment satisfying the above properties. 
\end{adf}}

	\section{Comparison with the motivic Chern classes} \label{s:mC}
	In this section we aim to compare the stable envelopes for \red{the} trivial slope with the motivic Chern classes of BB-cells. 
	Our main results are
\begin{pro} \label{pro:mC}
	Let $M$ be a projective, smooth variety
	\red{equipped} with an action of an algebraic torus $A$. \red{Suppose that the fixed point set $M^A$ is finite.} \old{with a finite number
	 of fixed points.} Consider the variety $X=T^*M$ \red{equipped} with the action of the torus $\T=\C^*\times A$. Choose any weight chamber $\mathfrak{C}$ of the torus $A$, polarization $T^{1/2}=TM$ and the trivial line bundle $\theta$ as a slope.
	Then\red{,} the elements
	$$
	\frac{mC_{-y}^A(M^+_F \to M)}{y^{\dim M_F^+}} \in K^A(X)[y,y^{-1}] \simeq K^\T(X) \simeq K^\T(T^*X)
	$$
	satisfy the axioms {\bf b)} and {\bf c)} of the stable envelope \red{$Stab^\theta(F)$.}
\end{pro}

\begin{rem}
	In this proposition we \red{do not} assume that $M$ \red{satisfies} the local product condition or even that $X$ satisfies the $(\star)$ condition.
\end{rem}

\begin{atw} \label{tw:mC}
	Consider the situation such as in proposition \ref{pro:mC}.
	Suppose that the variety $M$ with the action of a one dimensional torus $\sigma \in \mathfrak{C}$ satisfies the local product condition
	(definition \ref{df:prd}). Then \red{the element}
	$$
	\frac{mC_{-y}^A(M^+_F \to M)}{y^{\dim M_F^+}} \in K^A(X)[y,y^{-1}] \simeq K^\T(X) \simeq K^\T(T^*X)
	$$
	\old{determine}\red{is equal to} the $K$-theoretic stable envelope \red{$Stab^\theta(F)$.} \old{$y^{-\frac{1}{2}\dim M_F^+}Stab^\theta_{\mathfrak{C},T^{1/2}}(F)$.}
\end{atw}
\red{
Our main examples of varieties satisfying the local product property are homogenous spaces (see appendix \red{\hyperref[s:G/P]{B}}). Let $G$ be a reductive, complex Lie group with a chosen maximal torus $A$. Let $B$ be a Borel subgroup and $P$ a parabolic subgroup. We consider the action of the torus $A$ on the variety $G/P$.}

\red{
A choice of weight chamber $\mathfrak{C} \subset\mathfrak{a}$ induces a choice of Borel subgroup $B_\mathfrak{C}\subset G$. Let $F\in (G/P)^A$ be a fixed point.  It is a classical fact that the BB-cell $(G/P)_F^+$ (with respect to the chamber $\mathfrak{C}$)
coincides with the $B_\mathfrak{C}$-orbit of $F$. These orbits are called Schubert cells.
}
\begin{cor}
	\red{In the situation presented above} the stable envelopes for \red{the} trivial slope are equal to the motivic Chern classes of Schubert cells
	$$\frac{mC_{-y}^A((G/P)^+_F \to G/P)}{y^{\dim (G/P)^+_F}} 
	=y^{-\frac{1}{2}\dim (G/P)^+_F}Stab^\theta_{\mathfrak{C},T(G/P)}(1_{F}).$$
\end{cor}
\begin{proof}
	Theorem \ref{tw:prod} \red{implies}
	 that homogenous varieties satisfy the local product condition.
	 Thus, the corollary follows from theorem \ref{tw:mC}.
\end{proof}

\begin{rem}
	In the case of flag varieties $G/B$\red{,} our results for \red{the} trivial slope agree with the previous results of \cite{AMSS} (theorem 7.5) for a small anti-ample slope 
	up to \red{a} change of $y$ to $y^{-1}$. \red{This} difference is a consequence of the fact that in \cite{AMSS} the inverse action of $\C^*$ on the fibers of cotangent bundle is considered.
\end{rem}

Before the proof of theorem \ref{tw:mC} we make several simple observations.
	Let $\tilde{\nu}_F \simeq TM_{|F}$  denote the normal space to the fixed point $F$ in $M$.
	Denote by $\tilde{\nu}^-_F$ and $\tilde{\nu}^+_F$ its decomposition into the positive and negative part induced by the weight chamber $\mathfrak{C}$.
	Let
	$$\nu_F\red{\simeq TX_{|F}} \simeq TM_{|F}\oplus (T^*M_{|F}\red{\otimes \C_y})$$
	 denote the  normal space to the fixed point $F$ in the variety $X=T^*M$. It is a straightforward observation that
	\begin{align*}
	&	\nu_F^-\simeq\tilde{\nu}^-_F\oplus y(\tilde{\nu}^+_F)^* \red{\,,} \\
	&	\nu_F^+\simeq\tilde{\nu}^+_F\oplus y(\tilde{\nu}^-_F)^* \red{\,,} \\
	&T^{1/2}_{F,>0}= \tilde{\nu}^+_F.
	\end{align*} 

	In the course of  proofs we use the following computation:
	\begin{alemat} \label{lem:comp}
		Let $V$ be a $\T$-vector space. We have an equality
		$$
		\frac{\lambda_{-1}\left(y^{-1}V\right)}{\det V} =\frac{\lambda_{-y}(V^*)}{(-y)^{\dim V}}
		$$
		in the $\T$-equivariant $K$-theory of a point.
	\end{alemat}
\begin{proof}
	Both sides of the formula are multiplicative with respect to the direct sums of $\T$-vector spaces. Every $\T$-vector space decomposes as a sum of one dimensional spaces, so it is enough to check the equality for $\dim V=1$. Then it simplifies to trivial form:
	$$\frac{1-\frac{\alpha}{y}}{\alpha}=\frac{1-\frac{y}{\alpha}}{-y} ,$$
	where $\alpha$ is \red{the} character of the action of the torus $\T$ on the linear space $V$.
\end{proof}
\begin{proof}[Proof of proposition \ref{pro:mC}]
	We start the proof by checking the axiom {\bf b)}. We need to show that
	\begin{align*}
		\frac{mC_{-y}^A(M^+_F \to M)_{|F}}{y^{\dim M_F^+}}=
		eu(\nu^-_F)\frac{(-1)^{\rank T^{1/2}_{F,>0}}}{\det T^{1/2}_{F,>0}}	\,.	
	\end{align*}
	which is equivalent to
	
	\begin{align} \label{wyr:b}
	\frac{mC_{-y}^A(M^+_F \to M)_{|F}}
	{(-y)^{\dim \tilde{\nu}^+_F}}=
	eu(\tilde{\nu}^{-}_F)
	\frac{\lambda_{-1}(y^{-1}\tilde{\nu}^{+}_F)}{\det  \tilde{\nu}^+_F} \,.
	\end{align}
	
The BB-cell $M_F^+$ is a locally closed subvariety. 
Choose an open neighbourhood $U$ of the fixed point $F$ in $M$ such that the morphism $M_F^+\cap U \subset U$ is a closed immersion. The functorial properties of the motivic Chern class (cf. paragraph 2.3 of \cite{FRW}, or theorem 4.2 from \cite{AMSS})
 imply that:
\begin{align*}
&mC_{-y}^A\left(M^+_F \xto{i} M\right)_{|F}=
mC_{-y}^A\left(M^+_F \cap U \xto{i} U\right)_{|F}= \\
&=i_*mC_{-y}^A\left(id_{M^+_F \cap U}\right)_{|F}=
i_*\left(\lambda_{-y}\left(T^*(M^+_F \cap U)\right)\right)_{|F}=
\lambda_{-y}\left(\tilde{\nu}^{+*}_F\right)eu(\tilde{\nu}^-_F) \,.
\end{align*}
 
 So the left hand side of expression (\ref{wyr:b}) is equal to
 $$eu(\tilde{\nu}^-_F) \frac{\lambda_{-y}\left(\tilde{\nu}^{+*}_F\right)}{(-y)^{\dim \tilde{\nu}^{+}_F}}.$$
 Lemma \ref{lem:comp} implies that the right hand \red{side}
 is also of this form.
 
 We proceed to the axiom {\bf c)}. Consider a pair of fixed points $F,F'$ such that $F'<F$. We need to prove the inclusion:
 \begin{align} \label{wyr:Ninc}
 	 N^A\left(\frac{mC_{-y}^A(M^+_F \to M)_{|F'}}{y^{\dim M_F^+}}\right)
 	\subseteq
 	N^A(eu(\nu^-_{F'}))-\det T^{1/2}_{F',>0}
 \end{align}
 and take care of the distinguished point
 \begin{align} \label{wyr:Npt}
 	-\det T^{1/2}_{F',>0} \notin N^\red{A}\left(\frac{mC_{-y}^A(M^+_F \to M)_{|F'}}{y^{\dim M_F^+}}\right).
 \end{align}
Let's concentrate on the inclusion  (\ref{wyr:Ninc}).
 There is an equality of polytopes
 $$
 N^A(eu(\nu^-_{F'}))-\det T^{1/2}_{F',>0}=
 N^A\left(eu(\tilde{\nu}^{-}_{\red{F'}})
 \frac{\lambda_{-1}(y^{-1}\tilde{\nu}^{+}_{\red{F'}})}{\det  \tilde{\nu}^+_{\red{F'}}}\right)=
 N^A\left(eu(\tilde{\nu}^{-}_{F'})\lambda_{-y}\left(\tilde{\nu}^{+*}_{F'}\right)\right),
 $$
 where the second equality follows from lemma \ref{lem:comp}.
 After substitution of $y=1$ \red{into} the class $eu(\tilde{\nu}^{-}_{F'})\lambda_{-y}\left(\tilde{\nu}^{+*}_{F'}\right)$ we obtain the class $eu(\tilde{\nu}_{F'})$.
 Thus, proposition \ref{lem:New} (e) implies that \old{there is an inclusion}
 $$
 N^A(eu(\tilde{\nu}_{F'})) \subseteq N^A\left(eu(\tilde{\nu}^{-}_{F'})\lambda_{-y}\left(\tilde{\nu}^{+*}_{F'}\right)\right)\,.
 $$
Moreover
$$N^A\left(\frac{mC_{-y}^A(M^+_F \to M)_{|F'}}{y^{\dim M_F^+}}\right)= N^A\left(mC_{-y}^A(M^+_F \to M)_{|F'}\right)=
N^A\left(mC_{y}^A(M^+_F \to M)_{|F'}\right) \,.$$
Theorem 4.2 from \cite{FRW} implies that there is an inclusion 
$$
N^A\left(mC_{y}^A(M^+_F \to M)_{|F'}\right) \subseteq
N^A(eu(\tilde{\nu}_{F'})) \,.
$$
To conclude\red{,} we have proven inclusions
\begin{multline*}
N^A\left(mC_{y}^A(M^+_F \to M)_{|F'}\right) \subseteq
N^A(eu(\tilde{\nu}_{F'}))\subseteq \\
\subseteq N^A\left(eu(\tilde{\nu}^{-}_{F'})\lambda_{-y}\left(\tilde{\nu}^{+*}_{F'}\right)\right)=
N^A(eu(\nu^-_{F'}))-\det T^{1/2}_{F',>0} \,.
\end{multline*}

The next step is the proof of the formula (\ref{wyr:Npt}). We proceed in a manner similar to the proof of corollary 4.5 in \cite{FRW}. Consider a general enough one dimensional subtorus $\sigma\in \mathfrak{C}$. Proposition \ref{lem:Npi} implies that 
\begin{align} \label{wyr:gens}
N^\sigma\left(mC_{y}^A(M^+_F \to M)_{|F'}|_\sigma\right)=\pi_\sigma\left(N^A\left(mC_{y}^A(M^+_F \to M)_{|F'}\right)\right).
\end{align}

Theorem 4.2 from \cite{FRW} (cf. also theorem 10 from \cite{WeBB}), with the limit in $\infty$ changed to the limit in $0$ to get positive BB-cell instead of negative one, implies that 
$$
\lim_{\xi\to 0}\left(\left.\frac{mC_y^A(M_F^+\subset M)_{|F'}}{eu(\red{\tilde{\nu}}_{F'})}\right|_\sigma\right)=\;\chi_y(M_F^+\cap M^+_{F'})=\chi_y(\varnothing)=0\,,
$$
	\red{where $\xi$ is the chosen primitive character of the torus $\sigma$} and the class $\chi_y$  is the Hirzebruch genus (cf. \cite{chiy, BSY}).

Thus\red{,} the lowest term of line segment $N^\sigma\left(mC_{y}^A(M^+_F \to X)_{|F'}|_\sigma\right)$ is greater than the lowest term of line segment $ N^\sigma\left(eu(\red{\tilde{\nu}}_{F'})|_\sigma\right)$,
 which is equal to $\pi_\sigma(-\det T^{1/2}_{F',>0})$ by proposition \ref{lem:ver}.
Thus
$$\pi_\sigma(-\det T^{1/2}_{F',>0}) \notin \pi_\sigma\left(N^A\left(mC_{y}^A(M^+_F \to X)_{|F'}\right)\right).$$
\red{This}
implies
$$-\det T^{1/2}_{F',>0} \notin N^A\left(\frac{mC_{-y}^A(M^+_F \to X)_{|F'}}{y^{\dim M_F^+}}\right),$$
as demanded in (\ref{wyr:Npt}).
\end{proof}

To prove theorem \ref{tw:mC} we need the following technical lemma.
\begin{alemat}[cf. {\cite[Remark after Theorem 3.1]{RTV'}} {\cite[Lemma 5.2-4]{RTV}}] \label{lem:supp}
	Let $M$ and $X$ be varieties such as in proposition \ref{pro:mC}. Suppose that \red{$X$}
	satisfies the ($\star$) condition.
	Consider a fixed point $F \in M^A$.
	Suppose that an element $a \in K^\T(X)$ satisfies two conditions:
	\begin{enumerate}
		\item $\supp(a) \subset \bigcup_{F'\le F} T^*M_{|M^+_{F'}} \red{\,,}$
		\item  $\lambda_{-y}(T^*M^+_{F_i})_{|F_i}$ divides $a_{|F_i}$ for any fixed points $F_i \in M^A$.
	\end{enumerate}
	Then $\supp(a) \subset \bigcup_{F'\le F} \red{\nu^*}(M^+_{F'}\subset M)$.
\end{alemat}
\begin{proof}
	Consider the set of positive BB-cells of $M$ corresponding to fixed points $F'\le F$. Arrange them in a sequence $B_1,...,B_k$ in such a way that for every $t\le k$ the sum $\bigcup_{i=1}^t B_i$ is a closed subset of $M$ (cf. \cite{B-B3} theorem 3). Denote by $F_i$ the fixed point corresponding to the BB-cell $B_i$. Denote by
	$$E_t=\bigcup_{i=1}^t T^*M_{|B_i} \text{ and by }V=\bigcup_{i=1}^{k} \nu^*(B_i \subset M).$$
	Note that the ($\star$) condition implies that the sets $E_\red{t}\cup V$ are closed. Our goal is to prove by induction that
	$$\supp(a) \subset E_t \cup V.$$
	 The first condition implies this containment for $t=k$, which allows to start induction. To prove the proposition we need this containment for $t=0$. 

 	Assume that $\supp(a) \subset E_{t}\cup V$ for some $t \ge 1$. We want to prove that \hbox{$\supp(a) \subset E_{t-1}\cup V.$} 
 	Denote by $\iota$ the inclusion
 	$$\iota:E_{t}\cup V \subset X.$$
 	Element $a$ is equal to $\iota_{*}\alpha$ for some $\alpha \in G^\T(E_t \cup V)$. Denote by $U$ the variety $T^*M_{|B_t} \setminus \nu^*B_t$. Using the equality of the complements
 	$$\left(E_{t}\cup V\right) \setminus \left(E_{t-1}\cup V\right)=T^*M_{|B_t} \setminus \nu^*B_t=U\,,$$
 	we get an exact sequence
 	$$ G^\T(E_{t-1} \cup V) \to G^\T(E_t \cup V) \to G^\T(U) \to 0\,.$$
 	So it is enough to show that $\alpha$ restricted to the open subset $U$ vanishes. The variety $E_t \cup V$ is reducible. Denote by $i$ and $j$ the inclusions $E_t \subset E_t \cup V$ and $V \subset E_t \cup V$. Proposition \ref{cor:K} implies that the map
 	$$i_*+j_*: G^\T(E_t) \oplus G^\T(V) \onto G^\T(E_t \cup V)$$
 	is epimorphic. Choose any decomposition
 	$$\alpha=i_*\alpha_E +j_*\alpha_V$$
 	such that $\alpha_E \in G^\T(E_t)$ and $\alpha_V\in G^\T(V)$. The subsets $V$ and $U$ have empty intersection so $$(j_*\alpha_V)_{|U}=0.$$
 	Thus, it is enough to show that $i_*\alpha_E$ also vanishes after restriction to $U$.
 	
 	Note that lemma \ref{lem:comp} implies the following equality in $K^\T(F_t)$
 		$$\lambda_{-y}(T^*B_t)=\frac{(-y)^{\dim B_t}}{\det T^*B_t}\lambda_{-1}(y^{-1}TB_t)=\frac{(-y)^{\dim B_t}}{\det T^*B_t}eu(\nu^*B_t \subset T^*M_{|B_t}).$$
 		Moreover\red{,} the first map in the exact sequence of closed immersion
 		$$K^\T(\nu^*B_t) \xto{i_*} K^\T(T^*M_{|B_t}) \to K^\T(U) \to 0$$
 		is multiplication by the  Euler class $eu(\nu^*B_t \subset T^*M_{|B_i})$. It follows that for an arbitrary element $b\in K^\T(T^*M_{|B_i})$ the restriction of $b$ to the set $U$ is trivial if and only if $\lambda_{-y}(T^*B_t)_{\red{|F_t}}$ divides $b_{|F_t}$.
 	
 	 The second assumption and the fact that $(\iota_*j_*\alpha_V)_{|U}=0$ \red{imply} that the element 
 	$$(\iota_*i_*\alpha_E)_{|F_t}=a_{|F_t}-(\iota_{*}j_*(\alpha_V))_{|F_t}$$
 is divisible by $\lambda_{-y}(T^*B_t)_{\red{|F_t}}$.
 The pushforward-pullback argument shows that
 	$$\left(eu(T^*M_{|B_t}\subset X)\alpha_E\right)_{|F_t}=\left(\iota_*i_*\alpha_E\right)_{|F_t}.$$
 	It follows that $\lambda_{-y}(T^*B_i)_{|F_t}$ divides $\alpha_{E|F_t}$ multiplied by the Euler class \hbox{$eu(T^*M_{|B_t}\subset X)_{|F_t}$.} We need to prove that it divides $\alpha_{E|F_t}$. We use the following simple  algebra exercise:
 	\begin{exer*}
 		Let $R$ be a domain and $R[y,y^{-1}]$ the ring of Laurent polynomials. Assume that $A(y) \in R[y,y^{-1}]$ is a monic Laurent polynomials and $r\in R$ a nonzero element. Then for any polynomial $B(y) \in R[y,y^{-1}]$
 		 $$A(y)|B(y) \iff A(y)|rB(y) \,. $$
 	\end{exer*}
 	The ring $K^\T(F_t)$ is isomorphic to $K^A(F_t)[y,y^{-1}]$, the polynomial $\lambda_{-y}(T^*B_t)_{|F_t}$ is monic (the smallest coefficient is equal to one) and the Euler class $eu(T^*M_{|B_t}\subset X)_{|F_t}$ belongs to the subring $K^A(F_t)$. The exercise implies that $\lambda_{-y}(T^*B_t)$ divides $\alpha_{E|F_t}$. So \red{the class} $\alpha_E$ vanishes on $U$. Proof of the lemma follows by induction.
\end{proof}

\begin{proof}[Proof of theorem \ref{tw:mC}]
	Proposition \ref{pro:mC} implies that the axioms {\bf b)} and {\bf c)} holds. It is enough to check the support axiom.
	
	Chose a fixed point $F\in M^A$.	
	Functorial properties of the motivic Chern class imply that the support of class
	$$mC_{-y}^A(M^+_F \to M) \in K^A(M)[y]\subset K^\T(M)$$
	is contained in the closure of $M^+_F$ which is contained in the closed set
	$\bigcup_{F'\le F}M^+_{F'}.$
	It follows that the support of \red{the} pullback element
	$$mC_{-y}^A(M^+_F \to M) \in  K^\T(T^*M)$$ 
	is contained in restriction of the cotangent bundle $T^*M$ to the subset $\bigcup_{F'\le F}M^+_{F'}.$
	To prove the support axiom we need to check that it is contained in the smaller subset
	$$\bigcup_{F'\le F} \nu^*(M^+_{F'}\subset M).$$
	 Thus, it is enough to check  the assumptions of lemma \ref{lem:supp} for $a$ equal to the class $mC_{-y}^A(M^+_F \to M)$. We know that the first assumption holds. The local product condition (definition \ref{df:prd}) implies that for any fixed point $F'$
	 \begin{align*} mC_{-y}^A(M^+_{F} \to M)_{|F'}=&
	 mC_{-y}^A(M^+_F \cap U  \to U)_{|F'} \\
	 =&mC^A_{-y}(\red{Z_{F',F}'} \times M^+_{|F'} \to Z_{F'} \times M^+_{|F'})_{|F'} \\
	 =&mC^A_{-y}(\red{Z_{F',F}'} \subset Z_{F'})_{|F'}  mC^A_{-y}(M^+_{|F'} \to M^+_{|F'})_{|F'} \\
	 =&mC^A_{-y}(\red{Z_{F',F}'} \subset Z_{F'})_{|F'} \lambda_{-y}(T^*M_{F'}^+)_{|F'} \,.
	 \end{align*}
	\red{The second equality follows from the local product condition and the third from \cite[theorem 4.2(3)]{AMSS}.}
	Hence the class $mC_{-y}^A(M^+_F \to M)$ satisfies the assumptions of lemma \ref{lem:supp}.
\end{proof}

\begin{rem}
	In the proof of theorem \ref{tw:mC} we need the local product condition only to get divisibility demanded in lemma \ref{lem:supp}. Namely for a pair of fixed points $F,F'$ such that $F>F'$ we need divisibility of $mC_{-y}^{\red{A}}(M_F^+\subset M)_{|F'}$ by $\lambda_{-y}(T^*M^+_{F'})_{|F'}$.
	If the closure of the BB-cell of $F$ is smooth at $F'$ and the BB-cells form a stratification of $M$ then the divisibility condition automatically holds. In general case one can assume \red{the} existence of \red{a} motivically transversal slice instead of the local product condition.
Namely suppose there is a smooth locally closed subvariety $S \subset M$ such that:
	\begin{itemize}
		\item $F'\in S$ and $S$ is transversal to $M_{F'}^+$ at $F'$\red{,}
		\item $S$ is of dimension complementary to $M_{F'}^+$\red{,}
		\item $S$ is motivically transversal (cf. \cite{FRWp}, section 8) to $M_F^+$\red{.}
	\end{itemize}
	Then theorem 8.5 from \cite{FRWp}, or reasoning analogous to \red{the} proof
	of lemma 5.1 from \cite{FRW} proves the desired divisibility condition. In the case of homogenous varieties\red{,} divisibility  can be also acquired using theorem 5.3 of \cite{FRW} for the Borel group action.
\end{rem}

\section{Other slopes} \label{s:slope}
Computation of the stable envelopes for \red{the} trivial slope allows one to easily get formulas for all integral slopes.
\begin{cor}
	Consider a situation described in theorem \ref{tw:mC}. For a $A$-linearisable line bundle $s \in Pic(X)$ \red{the element}
	$$
	\frac{s}{s_{|F}}\cdot\frac{mC_{-y}^A(M^+_F \to M)}{y^{\dim M_F^+}} \in K^A(M)[y,y^{-1}] \simeq K^\T(M) \simeq K^\T(T^*M)
	$$
	\old{determine} \red{is equal to} the $K$-theoretic stable envelope \red{$Stab^s(F)$.}
	\old{$y^{-\frac{1}{2}\dim M_F^+}Stab^s_{\mathfrak{C},T^{1/2}}(F)$.}
\end{cor}
In this section we aim to prove that the stable envelope for \red{the} trivial slope coincides with the one for \red{a} sufficiently small anti-ample slope. Namely:
\begin{atw} \label{tw:slo}
	Let $M$ be a projective, smooth \red{A-variety. Suppose that the fixed point set $M^A$ is finite.}
	Consider the variety $X=T^*M$ with the action of the torus $\T=\C^*\times A$. 
	Choose any weight chamber $\mathfrak{C}$ of the torus $A$ and polarization $T^{1/2}=TM$.
	Suppose that $M$ satisfies the local product condition.
	For any anti-ample $A$-linearisable line bundle $s$ and a sufficiently big integer $n$\red{, the element}
	$$
	\frac{mC_{-y}^A(M^+_F \to M)}{y^{\dim M_F^+}} \in K^A(M)[y,y^{-1}] \simeq K^\T(M) \simeq K^\T(T^*M)
	$$
	\old{determine} \red{is equal to} the $K$-theoretic stable envelope \red{$Stab^{\frac{s}{n}}(F)$.}
	\old{$y^{-\frac{1}{2}\dim M_F^+}Stab^s_{\mathfrak{C},T^{1/2}}(F)$.}
\end{atw}
\begin{proof}
	Theorem \ref{tw:mC} implies that the considered \red{element satisfies} the axioms {\bf a)} and {\bf b)} of stable envelope. It is enough to check the axiom {\bf c)}. Namely for a fixed point $F'\le F$ we need to show that
	$$N^A\left(mC_{-y}^A(M^+_F \to M)_{|F'}\right) +\frac{s_{F}-s_{F'}}{n} \subseteq N^A(eu(\nu^-_{F'}) -\{0\})-\det T^{1/2}_{F',>0}.$$
	Note that a part of theorem \ref{tw:mC} is an analogous inclusion for the trivial slope. It implies that the point $-\det T^{1/2}_{F',>0}$ \red{does not} belong to the Newton polytope of motivic Chern class. The Newton polytope is a closed set\red{,} thus for \red{a} small enough vector
	 $\vv \in \mathfrak{a}^*$ its translation by $\vv$ also \red{does not} contain the point $-\det T^{1/2}_{F',>0}$.
	So it is enough to prove that there exists an integer $n$ such that
	$$N^A\left(mC_{-y}^A(M^+_F \to M)_{|F'}\right) +\frac{s_{F}-s_{F'}}{n} \subseteq N^A(eu(\nu^-_{F'}))-\det T^{1/2}_{F',>0}.$$
	Moreover\red{,} in the course of proof of proposition \ref{pro:mC} we showed containment of polytopes
	 $$N^A(eu(\tilde{\nu}^-_{F'}))\subseteq N^A(eu(\nu^-_{F'}))-\det T^{1/2}_{F',>0}.$$
	 Therefore, it is enough to show that for \red{a} big enough integer $n$ there is an inclusion
	 $$N^A\left(mC_{-y}^A(M^+_F \to M)_{|F'}\right) +\frac{s_{F}-s_{F'}}{n} \subseteq N^A(eu(\tilde{\nu}^-_{F'})).$$
	 
	 Consider a lattice polytope $N \subset \mathfrak{a}^*$. Define a facet as a codimension one face.
	 Let integral hyperplane denote an affine subspace of codimension one, spanned by lattice points. Suppose that the whole interior of the polytope $N$ lies on one side of an integral hyperplane $H$. Denote by $E_H$ \red{half-space}
	 which is the closure of the component of complement of $H$ which contains the interior of $N$. 
	 
	 Suppose that the affine span of \red{a lattice polytope} $N$ is the whole ambient space. For a facet $\tau$ of $N$\red{,} let $H_\tau$ be an integral hyperplane which is the affine span of the face $\tau$. Note that
	 $$N=\bigcap_{\tau} E_{H_\tau} $$
	 where the intersection is indexed by \red{the set of} codimension one faces. 
	  
	 \red{A} similar argument
	 can \red{be}
	 applied to any\old{Newton} \red{lattice} polytope, not necessarily spanning the whole ambient space. Denote by aff$(-)$ the affine span operator.
	 For any facet $\tau$ choose an integral hyperplane $H_\tau$ such that 
	 $$H_\tau \cap \text{aff}(N)=\text{aff}(\tau) \,. $$
	Then
	$$N=\text{aff}(N)\cap \bigcap_{\tau} E_{H_\tau} $$
	Thus\red{,} to check containment in the polytope $N$ it is enough to check containment in finitely many integral \red{half-spaces}
	$E_{H_\tau}$ and \red{the} affine span of $N$.  We use this observation for $N$ equal to the Newton polytope $N^A(eu(\tilde{\nu}^-_{F'}))$.
	  
	  \red{Let $H$ be an integral hyperplane.} We say that \red{a} vector
	   $\vv \in \mathfrak{a}^*$ points to $E_H$ when addition of $\vv$ preserves $E_H$. Our strategy of the proof is to show that for an integral hyperplane $H_\tau$ corresponding to \red{a} facet
	   $\tau$ of the polytope $N^A(eu(\tilde{\nu}^-_{F'}))$ at least one of the following conditions holds: 
	  \begin{itemize}
	  	\item The intersection $N^A\left(mC_{-y}^A(M^+_F \to M)_{|F'}\right) \cap H_\tau$ is empty.
	  	\item The vector $s_F-s_{F'}$ points to $E_{H_\tau}$.
	  \end{itemize}
  	Moreover\red{,} if an integral hyperplane $H$ contains the whole polytope $N^A(eu(\tilde{\nu}^-_{F'}))$ then the addition of the vector $s_F-s_{F'}$ preserves $H$. 
  	
  	Note that the above facts are sufficient to prove the theorem. Namely proposition \ref{pro:mC} shows that the polytope $N^A\left(mC_{-y}^A(M^+_F \to M)_{|F'}\right)$ is contained \red{in}
  	 $N^A(eu(\tilde{\nu}^-_{F'}))$.
  	It follows that it lies inside $E_{H_\tau}$ for every facet $\tau$.
  	If the vector $s_F-s_{F'}$ points to $E_{H_\tau}$ then for every integer $n \in \N$
  	$$N^A\left(mC_{-y}^A(M^+_F \to M)_{|F'}\right)+\frac{s_{F}-s_{F'}}{n} \subset E_{H_\tau} \red{\,.} $$
  	On the other hand if the intersection $N^A\left(mC_{-y}^A(M^+_F \to M)_{|F'}\right) \cap \red{H_\tau}$
  	is empty then translation of the polytope $N^A\left(mC_{-y}^A(M^+_F \to M)_{|F'}\right)$ by a sufficiently small vector still lies in~$E_{H_\tau}$. There are only finitely many facets of \red{ the polytope $N^A(eu(\tilde{\nu}^-_{F'}))$} so
  	there exists \red{an integer} $n$ such that
  	$$N^A\left(mC_{-y}^A(M^+_F \to M)_{|F'}\right)+\frac{s_{F}-s_{F'}}{n} \subset \bigcap_{\tau}E_{H_\tau}.$$
  	Moreover\red{,} addition of the vector $s_{F}-s_{F'}$ preserves the affine span of $N^A(eu(\tilde{\nu}^-_{F'}))$. It follows that for \red{a sufficiently big integer} $n$ the desired inclusion holds.
  	
  	Let $H \subset \mathfrak{a}^*$ be any integral hyperplane. Denote by $\tilde{H}$ the vector space parallel to $H$ (i.e a hyperplane passing through $0$).
  	Consider the one dimensional subspace $$\mathfrak{h}=ker(\mathfrak{a} \onto \tilde{H}^*).$$
  	The hyperplane $H$ is integral so $\mathfrak{h}$ corresponds to a one dimensional subtorus $\sigma_H\subset A$.
  	Choose an isomorphism $\sigma_H \simeq \C^*$ such that the induced map
  	$$\pi_H:\mathfrak{a}^* \onto \mathfrak{a}^*/\tilde{H} \simeq \mathfrak{h}^* \simeq \R,$$
  	sends the vectors pointing to
  	$E_H$ to \red{the} non-negative numbers. Thus, the vector $s_F-s_{F'}$ points to $E_H$
  	if and only if
  	 $$\pi_H(s_F-s_{F'})\ge0.$$
  	
  	 The \red{choice of} isomorphism $\sigma_H \red{\simeq} \C^*$ corresponds to \red{the} choice of primitive character $\ttt$  of the torus $\sigma_H$.
  	 To study the intersection with hyperplane $H$ we use the limit technique with respect to the torus $\sigma_H$. We use the definition of limit map from \cite{Kon} definition 4.1. It is a map defined on a subring of \red{the} localised K-theory:
  	 $$\lim_{\ttt \to 0}: S_A^{-1}K^A(pt)[y,y^{-1}] \dashrightarrow S_{A/\sigma_H}^{-1}K^{A/\sigma_H}(pt)[y,y^{-1}]\,.$$
  	 The multiplicative system $S_A$ (respectively $S_{A/\sigma_H}$) \old{is equal to} \red{consists of} all nonzero elements of $K^A(pt)$ (respectively $K^{A/\sigma_H}(pt)$). We present a sketch of construction of the above map. Choose an isomorphism of tori $A \simeq \sigma_H\times A/\sigma_H$. It induces an isomorphism $K^A(pt) \simeq K^{A/\sigma_H}(pt)[\ttt,\ttt^{-1}]$. Then the limit map is defined on the subring $K^{A/\sigma_H}(pt)[\ttt][y,y^{-1}]$ by killing all positive powers of {\bf t}. For technical details and extension to the localised $K$-theory
  	 see \cite{Kon} section~4. 
  	 
  	 Let $H$ be a hyperplane corresponding to \red{a} facet of \red{the} polytope $N^A(eu(\tilde{\nu}^-_{F'})) \subset E_H$. Thus (for a more detailed discussion see remark \ref{rem:lim})
  	$$N^A\left(mC_{-y}^A(M^+_F \to M)_{|F'}\right) \cap H = \varnothing \iff
  	\lim_{\ttt \to 0} \frac{mC_{-y}^A(M^+_F \to M)_{|F'}}{eu(\tilde{\nu}^-_{F'})}=0\,.$$
  	Let $\tilde{F} \subset X^{\sigma_H}$ be a component of the fixed \red{point} set
  	 which contains $F'$. Proposition 4.3 and theorem 4.4 from \cite{Kon} \red{imply}
  	 that
  	\begin{multline*}
  		\lim_{\ttt \to 0} \frac{mC_{-y}^A(M^+_F \to M)_{|F'}}{eu(\tilde{\nu}^-_{F'})}=
  		\lim_{\ttt \to 0}\left( \frac{mC_{-y}^A(M^+_F \to M)_{|\tilde{F}}}{eu(\tilde{F}\to M)\lambda_{-1}(T^*\tilde{F})}\right)_{|F'}= \\=
  		\left(\lim_{\ttt \to 0} \frac{mC_{-y}^A(M^+_F \to M)_{|\tilde{F}}}{eu(\tilde{F}\to M)\lambda_{-1}(T^*\tilde{F})}\right)_{|F'}= 
  		\left(\frac{1}{\lambda_{-1}(T^*\tilde{F})}\lim_{\ttt \to 0} \frac{mC_{-y}^A(M^+_F \to M)_{|\tilde{F}}}{eu(\tilde{F}\to M)}\right)_{|F'}= \\=
  		\left(\frac{mC_{-y}^{A/\sigma_{H}}(M^+_F \cap M_{\tilde{F}}^{\sigma_H,+}
  		\to \tilde{F})}{\lambda_{-1}(T^*\tilde{F})}\right)_{|F'}
  	\end{multline*} 
	where $M_{\tilde{F}}^{\sigma_H,+}$ is the positive BB-cell of $\tilde{F}$ with respect to the torus $\sigma_H$. It follows that if the intersection $ M^+_F \cap M_{\tilde{F}}^{\sigma_H,+}$ is empty then the intersection $N^A\left(mC_{-y}^A(M^+_F \to M)_{|F'}\right) \cap H$ is also empty.
	The closure of \red{the} set $M_{\tilde{F}}^{\sigma_H,+}$ is $A$-equivariant, thus $M^+_F \cap M_{\tilde{F}}^{\sigma_H,+}$ can be nonempty only if $F$ belongs to the closure of $M_{\tilde{F}}^{\sigma_H,+}$.
	To conclude it is enough to prove that $\pi_H(s_F-s_{F'})\ge 0$ whenever $F$ belongs to the closure of $M_{\tilde{F}}^{\sigma_H,+}$.
	
	For $F \in \overline{M_{\tilde{F}}^{\sigma_H,+}}$ there exist a finite number of points $A_1,...,A_{m-1},B_1,...,B_m \in M^{\sigma_H}$ such that
	 \begin{itemize}
	 	\item $F=B_1$ and $B_m \in \tilde{F}$,
	 	\item for every $i$ points $A_i$ and $B_i$ lies in the same component of \red{the fixed point set}~$M^{\sigma_H}$,
	 	\item there exists one dimensional $\sigma_H$-orbit from point $B_i$ to $A_{i-1}$,
	 \end{itemize}
 (see lemma 9 from \cite{B-B3} for a proof in the case of isolated fixed points). For a fixed point $B \in M^{\sigma_H}$ the fiber $s_{|B}$ is a $\sigma_H$-representation, denote by $\tilde{s}_B$ its  character.
 \red{If $B \in M^A$ then} there is an equality $\tilde{s}_B=\pi(s_B)$.
 The \red{line} bundle $s$ is anti-ample, so its restriction to every one dimensional $\sigma_H$ orbit is also anti-ample. For every anti-ample line bundle on $\PP^1$\red{,}  weight on the repelling fixed point is greater or equal than  weight on the attracting one\red{.} Thus
 $$\tilde{s}_{B_i} =\tilde{s}_{A_{i}} \ge \tilde{s}_{B_{i+1}}\,.$$
 It follows that
 $$\pi(s_F)=\tilde{s}_{B_1} \ge \tilde{s}_{B_m}=\pi(s_{F'}).$$
 
 Assume now that an integral hyperplane $H$ contains the whole polytope $N^A(eu(\tilde{\nu}^-_{F'}))$. Then the torus $\sigma_H$ acts trivially on the tangent space $\tilde{\nu}^-_{F'}$.
 Consider the fixed \red{point} set
  component $F_H\subset M^{\sigma_H}$ which contains $F$. It is a smooth closed subvariety of $M$ whose tangent space at $F'$ is the whole tangent space $T_{F'}M$. Thus\red{,} it contains the connected component of $F'$.
 It follows that $\sigma_H$ weights of $s$ restricted to $F$ and $F'$ coincide, thus $\pi(s_F-s_{F'})=0$.
\end{proof}
\red{
\begin{rem}
	 Let $H$ be a hyperplane corresponding to a facet of the polytope $N^A(eu(\tilde{\nu}^-_{F'}))$ and let $\sigma_H\subset A$ be the corresponding one dimensional subtorus. The  fixed point set $M^{\sigma_H}$ may be non-isolated.
\end{rem}
}
\begin{rem}
	The inequality about weights of \red{an} anti-ample line bundle on $\PP^1$ can be checked directly using the fact that all anti-ample line bundles on the projective line $\PP^1$ are of the form $\mathcal{O}(n)$ for $n < 0$. It can be also derived from the localization formula in equivariant cohomology and the fact that \red{the} degree of an anti-ample line bundle on $\PP^1$ is negative (cf. \cite{OM} paragraph 3.2.4).
\end{rem}

\begin{rem}\label{rem:lim}
	 \old{Intuitively one can think of the limit map of fraction as considering only this part of classes which corresponds to the last translation of hyperplane which intersects the Newton polytope of denominator. Namely} Consider the subtorus $\sigma_H \subset A$, corresponding to some integral hyperplane $H$, with chosen primitive character ${\bf t}$. The choice of ${\bf t}$ induces \red{a} choice of a \red{half-space}
	 $E_H$. Consider classes $a,b \in K^A(pt)$ such that 
	 $$N^A(a), N^A(b) \subset E_H \text{ and } N^A(b)\cap H\neq \varnothing.$$
	 \red{Intuitively, the limit of fraction $\lim_{\ttt \to 0}\frac{a}{b}$ takes into account the parts of classes $a,b$ that correspond to the intersections $N^A(a)\cap H$ and $N^A(b)\cap H$. More formally, a} choice
	 of splitting $A \simeq \sigma_H \times A/\sigma_H$ corresponds to a choice of integral character $\gamma \in \mathfrak{a}^*$ \red{whose} restriction to the subtorus $\sigma_H$ is equal to {\bf t}.
	 There \red{exists an} integer $m$ such that \red{$0\in\gamma^{m}H$ (thus $\gamma^{m}H=\tilde{H}$)}. It follows that under the isomorphism $K^A(pt)\simeq K^{A/\sigma_H}(pt)[\gamma,\gamma^{-1}]$ we have
	 $$\gamma^{m}a,\gamma^mb \in K^{A/\sigma_H}(pt)[\gamma].$$
	 Moreover\red{,} $\gamma^mb$ has nontrivial coefficient corresponding to $\gamma^{0}$. It is equal to the part of class $b$
	 \red{corresponding}
	 to the intersection $H\cap N^A(b)$. Denote by $q$ the projection
	 $$\red{q:}K^{A/\sigma_H}(pt)[\gamma] \to K^{A/\sigma_H}(pt) $$
	 defined by \red{$q(\gamma)=0$.}
	 It follows that the limit map is defined on the element $\frac{a}{b}$ as
	 $$\lim_{\ttt\to 0}\frac{a}{b}=\lim_{\ttt\to 0}\frac{\gamma^ma}{\gamma^mb}=\frac{q(\gamma^ma)}{q(\gamma^mb)}.$$
\end{rem}

\begin{rem}
	Using limit techniques one usually restricts to $K^\sigma(pt) \simeq \Z[\ttt,\ttt^{-1}]$ for \red{a} general enough subtorus $\sigma$ and then consider limits (cf. \cite{FRW,SZZ}).
	In our case\red{,} we consider a chosen subtorus $\sigma_H$ so we cannot proceed in this manner. It may happen that after restriction to $K^{\sigma_H}(pt)$ denominator vanishes.
\end{rem}

\red{\begin{rem}
	For a generalization of theorem \ref{tw:slo} to the case of arbitrary slope see our next paper \cite{KonW}.
\end{rem}}

\section{Example: The Projective plane}
In this section we aim to illustrate the proof of theorem \ref{tw:slo} by presenting explicit computations in the case of projective plane. We consider (using notation from theorem~\ref{tw:slo})
\begin{itemize}
	\item The torus $A=(\C^*)^2$ acting on the projective plane $M=\PP^2$ by:
	$$(t_1,t_2)[x:y:z]=[x:t_1y:t_2z].$$
	\item The weight chamber corresponding to the one dimensional subgroup $\sigma(t)=(t,t^2).$
	\item The anti-ample line bundle $s=\mathcal{O}(-1)$ as a slope.
	\item The fixed points $F=[0:1:0]$ and $F'=[0:0:1].$
\end{itemize}
Denote by $\alpha$ and $\beta$ characters of the torus $\T$ given by projections to the first and the second coordinates of $A$, respectively. Local computation leads to formulas:
\begin{multicols}{2}
	\begin{align*}
	&eu(\nu(F'\to M))_{|F'}=\left(1-\frac{\beta}{\alpha}\right)(1-\beta) \\
	&mC_{-y}(M^+_F\to M)_{|F'}=(1-y)\frac{\beta}{\alpha}(1-\beta) \\
	&\det T^{1/2}_{F',>0}=0 \\
	&\vv:=s_{|F}-s_{|F'}=\frac{\alpha}{\beta}
	\end{align*}
	\columnbreak
	
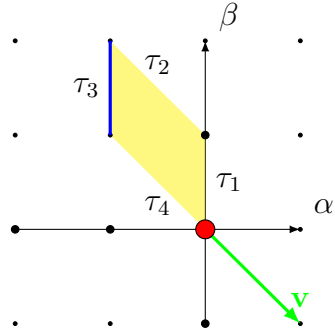
\centering
	\begin{tikzpicture}[scale=1.25]
	\coordinate (Origin)   at (0,0);
	\coordinate (XAxisMin) at (-2,0);
	\coordinate (XAxisMax) at (1,0);
	\coordinate (YAxisMin) at (0,-1);
	\coordinate (YAxisMax) at (0,2);
	
	\draw [thin, black,-latex] (XAxisMin) -- (XAxisMax);
	\draw [thin, black,-latex] (YAxisMin) -- (YAxisMax);
	
	\coordinate (B1) at (0,0);
	\coordinate (B2) at (0,1);
	\coordinate (B3) at (-1,2); 
	\coordinate (B4) at (-1,1);        

	\filldraw[fill=yellow, fill opacity=0.5, draw=yellow, draw opacity = 0] (B1)--(B2)--(B3)--(B4);
	
	\foreach \x in {-2,...,0}{
		\node[draw,circle,inner sep=1pt,fill] at (\x, 0) {};
	}
		\foreach \y in {-1,...,1}{
		\node[draw,circle,inner sep=1pt,fill] at (0, \y) {};
	}
	\foreach \x in {-2,...,1}{
		\foreach \y in {-1,...,2}{
			\node[draw,circle,inner sep=0.5pt,fill] at (\x,\y) {};
		}
	}
	
	\draw [very thick,blue](B3) -- (B4);
	\draw [very thick,green,-latex](0,0) -- (1,-1);
	\draw[ fill=red]   (0,0) circle (.1);
	\node at (1.25,0.25){$\alpha$};
	\node at (0.25,2.25){$\beta$};
	\node[green] at (1,-0.75) {$\vv$};	
	\node at (0.25,0.5){$\tau_1$};
	\node at (-0.5,1.75){$\tau_2$};
	\node at (-1.25,1.5){$\tau_3$};
	\node at (-0.5,0.25){$\tau_4$};
	\end{tikzpicture}
\end{multicols}
Denote by $B:=N^A(mC_{-y}(M^+_F\to M)_{|F'})$ (blue interval), $C=N^A(eu(\nu(F'\to M))_{F'})$ (yellow parallelogram) and $D=\det T^{1/2}_{F',>0}=0$ (red point). Denote the facets of polytope $C$ by $\tau_1,\tau_2,\tau_3, \tau_4$ according to the picture.

Theorem \ref{tw:mC} implies that the blue interval $B$ is contained in the yellow polytope $C$ and the red point $D$ \red{does not} belong to the interval $B$.
It is enough to prove that for \red{a} sufficiently big integer $n$
and every facet $\tau$
$$B+\frac{\vv}{n}\subset E_{H_{\tau}}.$$
Let's compute
half-planes
and subtori associated with the facets:
$$\begin{array}{|c|c|c|c|c|c|} \hline
	\text{facet} & E_{H_\tau} & \tilde{H}_\tau &\sigma_H\subset A & \pi_{H_\tau} & \text{character {\bf t}} \\
	\hline
	\tau_1 & \{x\alpha+y\beta|x\le 0\}&\red{\lin}(\beta)&(t,0)&x\alpha+y\beta \to -x&(t,0)\to \frac{1}{t}\\
	\tau_2 & \{x\alpha+y\beta|x+y\le 1\}&\red{\lin}(\alpha-\beta)&(t,t)&x\alpha+y\beta \to -x-y&(t,t)\to \frac{1}{t}\\
	\tau_3 & \{x\alpha+y\beta|x\ge -1\}&\red{\lin}(\beta)&(t,0)&x\alpha+y\beta \to x&(t,0)\to t\\
	\tau_4 & \{x\alpha+y\beta|x+y\ge 0\}&\red{\lin}(\alpha-\beta)&(t,t)&x\alpha+y\beta \to x+y&(t,t)\to t\\ \hline
\end{array}$$
\red{Where $\lin$ denotes the linear span.}
\red{A} choice
of splitting $A \simeq \sigma_H \times A/\sigma_H$ corresponds to a choice of integral character $\gamma \in \mathfrak{a}^*$
\red{whose}
restriction to the subtorus $\sigma_H$ is equal to {\bf t}. For $\tau_1$ and $\tau_3$ let's choose  $\gamma=\left(\frac{\alpha}{\beta^2}\right)^{\pm 1}.$ It induces a splitting of cohomology $K^A(pt)=\Z[\beta,\beta^{-1}][\frac{\alpha}{\beta^2},\frac{\beta^2}{\alpha}].$ Using this splitting we can compute limits
\begin{align*}
	&\tau_1: \ \lim_{\ttt \to 0}(1-y) \frac{\frac{\beta}{\alpha}}{1-\frac{\beta}{\alpha}}=
	(1-y)\lim_{\ttt \to 0} \frac{\frac{\beta^2}{\alpha}}{\beta-\frac{\beta^2}{\alpha}}=(1-y)\frac{0}{\beta}=0 \red{\,,} \\
	&\tau_3: \ \lim_{\ttt \to 0}(1-y) \frac{\frac{\beta}{\alpha}}{1-\frac{\beta}{\alpha}}=
	(1-y)\lim_{\ttt \to 0} \frac{\frac{1}{\beta}}{\frac{\alpha}{\beta^2}-\frac{1}{\beta}}=(1-y)\frac{\frac{1}{\beta}}{-\frac{1}{\beta}}=y-1 \red{\,.}
\end{align*}
	 For $\tau_2$ and $\tau_4$ let's choose $\gamma=\alpha^{\pm 1}$. It induces a splitting $K^A(pt)=\Z[\frac{\alpha}{\beta},\frac{\beta}{\alpha}][\alpha,\alpha^{-1}]$ (the character $\frac{\beta}{\alpha}$ is a basis of $\tilde{H}_\tau$) and
	\begin{align*}
	 &\tau_2,\tau_4: \lim_{\ttt \to 0}(1-y) \frac{\frac{\beta}{\alpha}}{1-\frac{\beta}{\alpha}}=
	 (1-y)\frac{\frac{\beta}{\alpha}}{1-\frac{\beta}{\alpha}}.
	\end{align*}
	 \red{These}
	 calculations imply\old{the fact} that  
	 $$B\cap\tau_i \neq \varnothing \iff i\in\{2,3,4\}.$$
	 We want to show that \red{the} addition of the vector $\vv$ preserves
	 half-plane
	 $E_{H_{\tau}}$ for these three facets by proving that
	$$\pi_{H_{\tau_i}}(\vv) \ge 0 \text{ for } i \in \{2,3,4\}. $$ 
	For $\tau_2,\tau_4$ the points $F',F$ \red{belong} to the same fixed \red{point} set
	component of the torus~$\sigma_{H}$. It implies that $\pi_{H}(\vv)=0$\red{. This} agrees with direct computation
	$$\pi_{H}(\vv)=\pi_H(\alpha-\beta)=\pm(1+(-1))=0.$$
	Moreover\red{,} for $\tau_3$ there is one dimensional $\sigma_H$-orbit $[0:x:y]$ from $F$ to $F'$. It implies that $\pi_{H_{\tau_3}}(\vv)>0$\red{. This} agrees with direct computation
	$$\pi_{H_{\tau_3}}(\vv)=\pi_{H_{\tau_3}}(\alpha-\beta)=1.$$
	To conclude, for every $n \in \N$ the interval $B+\frac{\vv}{n}$ is contained in the intersection of
	half-planes
	 $\bigcap_{i=2}^4E_{H_{\tau_i}}$. Moreover\red{,} it is also contained in $E_{H_{\tau_1}}$ for \red{a} sufficiently big integer~$n$.
	
\section{Appendix \red{A}: uniqueness of the stable envelopes} \label{s:Ok}
In \cite{OS,O2} the stable envelope was defined for \red{an} action of a reductive group $G$. In this appendix we show that for the group $G$ equal to a torus and a general enough slope our definition \ref{df:env} of the stable envelope coincides with\old{the} Okounkov's\old{one}. Moreover\red{,} we prove the uniqueness of stable envelopes for an arbitrary slope.

	\red{We use the notations and assumptions from the beginning of section \ref{s:env}.} According to \cite{OS,O2} the stable envelope is a map $$K^\T(X^A) \to K^\T(X)$$ given by a correspondence
	$$Stab \in K^\T(X^A\times X) \,,$$
	which \red{satisfies} three properties (cf. \cite{O2} paragraph 9.1.3).
	For a $\T$-variety $X$ and a finite set  $F$ \red{(with the trivial $\T$-action)} any map of $K^\T(pt)$ modules
	$$f: K^\T(F) \to K^\T(X)$$
	is determined by a correspondence $G \in K^\T(F\times X)$ such that for any $x\in F$
	$$G_{x\times X}=f(1_{x}).$$
	Below we denote  both morphism and correspondence by $Stab$. The main ingredient in\old{the} Okounkov's definition are attracting sets. For a one parameter subgroup $\sigma:\C^*\to A$ it is defined as (cf. \cite{OS} paragraph 2.1.3, \cite{O2} paragraph 9.1.2):
	$$Attr=\{(y,x) \in X^A \times X|\lim_{t\to 0}\sigma(t)x=y\}  \,. $$
	Moreover\red{,} for \red{a fixed point $F\in X^A$} we define
	$$Attr(F)=\{x|\lim_{t\to 0}x\in F\} \subset X \,.$$
	The straightforward comparison of definitions shows that the attracting sets coincide with the BB-cells
	$$Attr(F)=X_F^+ \,.$$
	\old{Moreover in the case of isolated fixed points we have
	$Attr=\bigsqcup_{F \in X^\T} F \times X_F^+ \,.$}	

{\bf Support condition:} (paragraph 2.1.1 from \cite{OS},
paragraph 9.1.3 point 1 from \cite{O2}
and theorem 3.3.4 point (i) of \cite{OM})
In\old{the} Okounkov's papers it is required that
$$\supp (Stab) \subset \bigsqcup_{F \in X^\red{A}}
\left(F\times \bigsqcup_{F'<F}Attr(F') \right) \,,$$
which means exactly that for any fixed point $F \in X^A$
$$\supp(Stab(1_F)) \subset 
\bigsqcup_{F'<F}Attr(F')\,.$$
The attracting sets coincide with the BB-cells, so this is \red{an} equivalent formulation
 of the axiom {\bf a)}.

{\bf Normalization condition:}  (paragraph 2.1.4 from \cite{OS}, paragraph 9.1.5 from \cite{O2})
 Consider any fixed point $F \in X^\red{A}$.
 In\old{the} Okounkov's papers it is required that the correspondence $Stab \in K^\T(X^\red{A}\times X)$
 inducing the stable envelope morphism \red{satisfies}
$$Stab_{|F\times F}=(-1)^{\rank T^{1/2}_{>0}}\left(\frac{\det \nu^-_F}{\det T^{1/2}_{F,\neq0}}\right)^{1/2} \otimes \mathcal{O}_{Attr}|_{F\times F} \red{\,.}$$
After substitutions
\begin{align*}
&\mathcal{O}_{Attr}|_{F\times F}=\mathcal{O}_{\diag F} \otimes eu(\nu_F^-) \red{\,,} \\
& \frac{\det \nu^-_F}{\det T^{1/2}_{F,\neq0}}= h^{\rank T^{1/2}_{>0}}\left(\frac{1}{\det T^{1/2}_{F,>0}}\right)^2
\end{align*}
\red{as} noted in  paragraph 2.1.4 of \cite{OS} we obtain
$$Stab_{|F\times F}=
eu(\nu^-_F)\frac{(-1)^{\rank T^{1/2}_{F,>0}}}{\det T^{1/2}_{F,>0}} \otimes h^{\frac{1}{2}\rank T^{1/2}_{F,>0}}\otimes \mathcal{O}_{\diag F}.$$
Changing correspondence to a morphism we get an equivalent condition
$$Stab(1_F)_{|F}=
eu(\nu^-_F)\frac{(-1)^{\rank T^{1/2}_{F,>0}}}{\det T^{1/2}_{F,>0}} h^{\frac{1}{2}\rank T^{1/2}_{F,>0}}\,,$$
which is exactly our axiom {\bf b)}.

{\bf Smallness condition:} (paragraph 2.1.6 from \cite{OS}, paragraph 9.1.9 from \cite{O2})
In the case of isolated fixed points\red{,} the last axiom of stable envelope from\old{the} Okounkov's papers states that for any pair of fixed points $F_1,F_2 \in X^A$ 
 $$N^A\left(Stab_{|F_1\times F_2}\otimes s_{|F_1}\right) \subseteq
 N^A\left(Stab_{|F_2\times F_2}\otimes s_{|F_2}\right).$$
 The support condition implies that this requirement is nontrivial only when $F_1 > F_2.$ Changing correspondence to a morphism we get an equivalent form
 $$
 N^A\left(Stab(1_{F_1})_{|F_2}\right) +s_{|F_1} \subseteq
 N^A\left(Stab(1_{F_2})_{|F_2}\right) + s_{|F_2}.
 $$
 \old{Replace $Stab(1_{F_2})_{|F_2}$ by its value determined by the normalization condition.}
 \red{The normalization axiom implies that
 $$N^A\left(Stab(1_{F_2})_{|F_2}\right)=N^A\left(eu(\nu^-_{F_2})\frac{(-1)^{\rank T^{1/2}_{{F_2},>0}}}{\det T^{1/2}_{{F_2},>0}} h^{\frac{1}{2}\rank T^{1/2}_{{F_2},>0}}\right) \,. $$}
 Note that the torus $A$ preserves the symplectic form $\omega$, thus multiplication by $h$ \red{does not} change Newton polytope $N^A$. Thus, we get an equivalent formulation
 $$N^A\left(Stab(1_{F_1})_{|F_2}\right)+s_{|F_1}
 \subseteq
 N^A\left(eu(\nu^-_{F_2})\right) -\det T^{1/2}_{F_2,>0} +s_{|F_2},$$
 which is very similar to the axiom {\bf c)}. The only difference is that we additionally require
 $$-\det T^{1/2}_{F_2,>0} +s_{|F_2} \notin N^A\left(Stab(1_{F_1})_{|F_2}\right)+s_{|F_1} .$$
 For a general enough slope this requirement automatically holds because \red{the point}
 $$-\det T^{1/2}_{F_2,>0} +s_{|F_2}-s_{|F_1}$$
 is a vertex of polytope which \red{is not} a lattice point.
 \red{The} addition of this assumption is \red{necessary} to acquire uniqueness of the stable envelopes for all slopes.

	\begin{ex} \label{ex:uni}
	Consider the variety $X=T^*\PP^1$ \red{equipped} with the action of the torus $\T=\C^*\times A$ where $A$ is the one dimensional torus acting on $\PP^1$ by
	$$\alpha[a:b]=[\alpha a:b]$$
	and $\C^*$ acts on the fibers by scalar multiplication. Denote by $\alpha$ and $y$ characters of $\T$ corresponding to projections to the tori $A$ and $\C^*$. The action of the torus $\T$ has two fixed points $\ee_1=[1:0]$ and $\ee_2=[0:1]$. 
	The variety $X$ satisfies \red{the} condition~$(\star)$ in a trivial way.
	
	Consider the stable envelope for the positive weight chamber (such that $\alpha$ is \old{a} positive), the tangent bundle $T\PP^1$ as polarization and the trivial line bundle \red{$\theta$} as a slope. 
	If we omit the point zero in the axiom {\bf c)} then both
	$$Stab^\red{\theta}(\ee_1)=1-O(-1), \ Stab^\red{\theta}(\ee_2)=\frac{1}{y}- \frac{O(-1)}{\alpha}$$
	and $$Stab^\red{\theta}(\ee_1)=1-O(-1), \ Stab^\red{\theta}(\ee_2)=\frac{O(-1)}{y}- \frac{O(-2)}{\alpha}$$
	\red{satisfy} the axioms of stable envelope.
\end{ex}

The rest of this appendix is devoted to the proof of uniqueness of the stable envelope (proposition \ref{pro:uniq}).
For \red{a} general enough slope it was proved in proposition 9.2.2 of \cite{O2}. For the sake of completeness\red{,} we present it with all necessary technical details omitted in the original.
The proof is a generalisation of the proof of uniqueness of cohomological envelopes (paragraph 3.3.4 in \cite{OM}).
We need the following lemma.
\begin{alemat} \label{lem:uniq}
	Choose a set of vectors $l_F \in \Hom(A,\C^*)\otimes \Q$ indexed by \red{the fixed point set $X^A$.}
	Suppose that an element $a \in K^\T(X)$ satisfies conditions
	\begin{enumerate}
		\item $\supp(a) \subset \bigsqcup_{F\in X^\red{A}} X^+_F$\red{,}
		\item		\red{for any fixed point $F\in X^A$ we have containment of the Newton polytopes}
			$$N^A(a_{|F}) \subseteq \left(N^A(eu(\nu^-_{F}))\setminus \{0\}\right)+l_F \,.$$	
	\end{enumerate}
	Then $a=0$.
\end{alemat}
\begin{proof}
	We proceed by induction on the partially ordered set $X^A$. 
	Suppose that the element $a$ is supported on the closed set $Y=\bigsqcup_{F\in Z} X^+_F$ for some subset $Z\subset X^\red{A}$. Choose a BB-cell $X^+_{F_1}$, corresponding to fixed point $F_1 \in X^A$, which is an open subvariety of $Y$. We aim to show that $a$ is supported on the closed subset $\bigsqcup_{F\in (Z-F_1)} X^+_F$. By induction it implies that $a=0$. \\
	Choose an open subset $U \subset X$ such that $U \cap Y =X^+_{F_1}$. Consider the diagram 
	$$
	\xymatrix{
		& U \ar[r]^i & X \\
		F_1 \ar[r]^{s_0} & X^+_{F_1} \ar[u]^{\tilde{j}} \ar[r]^{\tilde{i}} & Y \ar[u]^{j}
	}
	$$
	The square in the diagram is \red{a} pullback. The BB-cells are smooth\red{,} locally closed subvarieties, so the map $\tilde{j}$ is an inclusion of \red{a} smooth subvariety.
	There exist an element $\alpha \in G^\T(Y) $ such that $j_*(\alpha)=a$. 
	It follows that:
	$$ a_{|F_1}=(j_*\alpha)_{|F_1}=
	\red{s_0^*}\tilde{j}^*i^*j_*\alpha=
	\red{s_0^*}\tilde{j}^*\tilde{j}_*\tilde{i}^* \alpha = eu(\nu_{F_1}^-) \alpha_{|F_1},$$
	which implies
	\begin{align} \label{wyr:i1}
	N^A(eu(\nu_{F_1}^-) \alpha_{|F_1}) =N^A(a_{|F_1}) \subseteq
	\left(N^A(eu(\nu^-_{F_1}))\setminus \{0\}\right) +l_{F_1}.
	\end{align}
	Assume that $\alpha_{|F_1}$ is a nonzero element. Then the Newton polytope $N^A(\alpha_{|F_1})$ is nonempty. The ring $K^{\T/A}(F_1)$ is a domain so  proposition \ref{lem:New} (b) implies that
	\begin{align} \label{wyr:i2}
	N^A\left(eu(\nu^-_{F_1})\right) \subseteq
	N^A\left(eu(\nu_{F_1}^-)\right) +N^A(\alpha_{|F_1})=
	N^A\left(eu(\nu_{F_1}^-) \alpha_{|F_1}\right).
	\end{align}
	\old{In the case of non isolated fixed points one need to use proposition \ref{lem:New} (c) for the class $eu(\nu^-_{F_1})$ (see remark \ref{rem:N}).} The inclusions (\ref{wyr:i1}) and (\ref{wyr:i2}) imply that
	$$N^A\left(eu(\nu^-_{F_1})\right)\subseteq\left(N^A(eu(\nu^-_{F_1}))\setminus \{0\}\right) +l_{F_1} \,. $$
	But no polytope can be translated into a proper subset of itself. This contradiction proves that the element $\alpha_{|F_1}$ is equal to zero. The map $s_0$ is a section of an affine bundle 
	so it induces an isomorphism \red{of} the algebraic $K$-theory. It follows that $\alpha_{|X^+_{F_1}}=0$. Thus\red{,}
	the element $\alpha$ is supported on the closed set $\bigsqcup_{F\in (Z-F_1)} X^+_F$. It follows that $a$ is also supported on this set.  
\end{proof}
\begin{proof}[Proof of proposition \ref{pro:uniq}]
	Let $\{Stab(F)\}_{F\in X^A} $ and $\{\widetilde{Stab}(F)\}_ {F\in X^A}$ be two sets of elements satisfying the axioms of stable envelope. It is enough to show that for any \red{fixed point $F\in X^A$} the element $Stab(F)-\widetilde{Stab}(F)$
	satisfies conditions of lemma  \ref{lem:uniq} for  the set of vectors
	$$l_{F'}= s_{F'}-s_F-\det T^{1/2}_{F',>0} .$$
	The support condition follows from the axiom {\bf {a)}}.
	Let's focus on the second condition. The only nontrivial case  is $F'< F$. In the other cases the axioms {\bf {a)}} and {\bf {b)}} imply that
	$$Stab(F)_{\red{|F'}} -\widetilde{Stab}(F)_{\red{|F'}}=0.$$
	When $F'< F$ the axiom {\bf {c)}} implies that the Newton polytopes $N^\red{A}(Stab(F)_{\red{|F'}})$ and $N^\red{A}(\widetilde{Stab}(F)_{\red{|F'}})$ are contained in the convex set (cf. proposition \ref{lem:ver})
	$$\left(N^A(eu(\nu^-_{\red{F'}}))\setminus \{0\}\right)+l_\red{F'}.$$
	Thus
	\begin{align*}
	N^\red{A}\left(Stab(F)_{\red{|F'}}-\widetilde{Stab}(F)_{\red{|F'}}\right) \subseteq&
	conv\left(N^\red{A}(Stab(F)_{\red{|F'}}),N^\red{A}(\widetilde{Stab}(F)_{\red{|F'}})\right) \\
	\subseteq& \left(N^A(eu(\nu^-_{\red{F'}}))\setminus \{0\}\right)+l_\red{F'} \,.
	\end{align*}  
\end{proof}
\red{
\begin{rem}
	In this paper we always assume that the fixed point set $X^A$ is finite. In the case of nonisolated fixed points, our definition \ref{def:ele} is not equivalent to Okounkov's definition. For a component of the fixed point set $F\subset X^A$, the morphism
	$$K^\T(F)\to K^\T(X)$$
	 is not determined by its value on the element $1_F$. However, even in this case,  there is at most one element satisfying the axioms of definition \ref{def:ele}. The proofs of analogues of proposition \ref{pro:uniq} and lemma \ref{lem:uniq} are almost identical to those presented above. The only difference is that the ring $K^\T(F)$ may not be a domain. Thus, in the proof of lemma \ref{lem:uniq}, one needs to use proposition \ref{lem:New} (c) (for the class $eu(\nu^-_{F_1})$, see remark \ref{rem:N}) instead  of \ref{lem:New} (b).
\end{rem}
}
\section{Appendix \red{B}: The local product property of Schubert cells} \label{s:G/P}
Let $G$ be \red{a} reductive, complex Lie group
with chosen maximal torus $\T$ and Borel subgroup $B^+$. Any one dimensional subtorus $\sigma \subset \T$ induces a linear functional $$\varphi_\sigma:\mathfrak{t}^* \to \C.$$ For a general enough subtorus $\sigma$ 
we can assume that no roots belong to the kernel of this functional. Consider the Borel subgroups $B_\sigma^+$ such that the corresponding Lie algebra is \red{the} union of these weight spaces whose characters are positive with respect to $\varphi_\sigma$. Denote its unipotent subgroup by $U_\sigma^+$. Analogously one can define groups $B_\sigma^-$ and $U_\sigma^-$.

For a parabolic group $B^+ \subset P \subset G$ consider the BB-decomposition of the variety $G/P$ with respect to the torus $\sigma$.
It is a classical fact that
the positive (respectively negative) BB-cells
are the orbits of  group $B_\sigma^+$ (respectively $B_\sigma^-$).
We prove that the stratification of $G/P$ by BB-cells of the torus $\sigma$ behaves like a product in a neighbourhood of a fixed point of the torus $\T$ (see definition \ref{df:prd}).
\begin{atw}\label{tw:prod}
	Consider the situation described above.
Any fixed point $x \in(G/P)^\T$ has an open neighbourhood $U$ such that:
	\begin{enumerate}
		\item There exist a $\T$-equivariant isomorphism $$\theta: U \simeq \left(U\cap (G/P)_x^+\right)\times \left(U\cap (G/P)_x^-\right)  $$
		\item For any fixed point $y \in(G/P)^\T$ the isomorphism $\theta$ induces isomorphism:
		$$U\cap (G/P)_y^+ \simeq \left(U\cap (G/P)_x^+\right)\times \left(U\cap (G/P)_x^- \cap (G/P)_y^+\right)$$
	\end{enumerate}	
\end{atw}
In the course of proof we use the following interpretation of classical notions of the theory of Lie groups in the language of BB-decomposition. Note that we consider BB-cells in smooth quasi-projective varieties.
\begin{alemat} \label{lem:BorABB}
	Consider \red{the} action
	of the torus $\sigma$ on the group $G$ defined by conjugation. Denote by $F$ the component of the fixed \red{point} set
	which contains the identity. For a subset $Y\subset F$ we use abbreviations
	$$G^+_Y=\{x\in G| \lim_{t\to 0}x\in Y\} \text{ and } G^-_Y=\{x\in G| \lim_{t\to \infty}x\in Y\}$$
	for the fibers of projections $G^+_F\to F$ and $G^-_F\to F$ over $Y$.
	\begin{enumerate}
		\item The Borel subgroup $B_\sigma^+$ (respectively $B_\sigma^-$) is the positive (respectively negative) BB-cell of the maximal torus $\T$ i.e.
		$$B_\sigma^+=G^+_\T\,.$$
		\item The unipotent subgroup $U_\sigma^+$ (respectively $U_\sigma^-$) is the positive (respectively negative) BB-cell of the identity element i.e.
		$$U_\sigma^+=G^+_{id}\,.$$
	\end{enumerate} 
\end{alemat}
\begin{proof}
	We prove only the first case for the positive Borel subgroup. \red{The} other cases
	are analogous.
	It is enough to show that $G^+_\T$ is a connected subgroup of $G$ whose Lie algebra coincides with the Lie algebra of $B_\sigma^+$.
	
	The variety $G^+_\T$ is a subgroup because the maximal torus $\T$ is a group and \red{the} limit \red{map}  preserves multiplication. Namely for $g,h\in G^+_\T$ such that
	$\lim_{t\to 0} g=a \in \T$ and $\lim_{t\to 0} h=b \in \T$
	it is true that
	$$\lim_{t\to 0} g^{-1}h=a^{-1}b \in \T. $$
	The variety $G^+_\T$ is connected because the maximal torus $\T$ is connected.
	So it is enough to compute the tangent space to $G^+_\T$ at identity. The exponent map is an isomorphism in some neighbourhood of zero so we can limit ourselves to computation in the Lie algebra $\mathfrak{g}$.
	The action of \red{the} torus $\sigma$ on $\mathfrak{g}$ is given by differentiation of the action on $G$.
	The tangent space $T_0G^+_\T$ is equal to the BB-cell $\mathfrak{g}^+_\mathfrak{t}$.
	Consider \red{the} weight decomposition $$\mathfrak{g}=\bigoplus_{h\in \mathfrak{t}^*} V_h.$$
	Differentiation of the action of $\sigma$ on $\mathfrak{g}$ is equal to the Lie bracket so
	$$\mathfrak{g}_\mathfrak{t}^+=\bigoplus_{h\in \mathfrak{t}^*, \varphi_\sigma(h) \ge 0} V_h. $$
	\red{This}
	is exactly the tangent space to the Borel subgroup $B_\sigma^+$. 
\end{proof}
\begin{proof}[Proof of the theorem \ref{tw:prod}]
	Note that the Weyl group acts transitively on \red{the fixed point set $(G/P)^\T$.}
	Thus\red{,} replacing the torus $\sigma$ by its conjugate by a Weyl group element we may assume that a fixed point $x$ is equal to the class of identity.
	
	Denote the Lie algebra  of the parabolic subgroup $P$ by $\mathfrak{p} \subset \mathfrak{g}$. Denote by $\mathfrak{u}_P$ the Lie subalgebra consisting of the root spaces which \red{do not} belong to $\mathfrak{p}$. Let $U_P$ be \red{the}
	corresponding Lie group.
	 The group $U_P$ is unipotent (as a subgroup of the unipotent group $U^-$). Consider the action of \red{the} torus
	 $\T$ on $U_P$ by conjugation. Let's note two facts from the theory of Lie groups.
	\begin{enumerate}
		\item $U_P$ is isomorphic to its complex Lie algebra as a complex $\T$-variety (cf. paragraph 15.3b from \cite{Bor}, or paragraph 8.0 from \cite{Unip}). \label{1}
		\item The quotient map $G \to G/P$ induces $\T$-equivariant isomorphism from $U_P$ to some open neighbourhood of identity.
	\end{enumerate}
	Choose $U_P$ as a neighbourhood $U$ of identity.
	The second observation and \red{the} second point of lemma \ref{lem:BorABB} \red{imply} that:
	$$ X_+:=U_P \cap (G/P)_{id}^+ \simeq U_P \cap G_{id}^+ \simeq U_P \cap U_\sigma^+ \red{\,,}$$
	analogously
	$$ X_-:=U_P \cap (G/P)_{id}^- \simeq U_P \cap U_\sigma^-.$$
	Both \red{isomorphisms} are given by the quotient morphism $G \to G/P$.  We define morphism 
	$$\theta: X_+ \times X_- \to U_P $$
	as multiplication in $U_P$. We aim to prove that this is an isomorphism. We start by showing injectivity on points. Both varieties $X_+$ and $X_-$ are subgroups of $U_P$. So to prove injectivity it is enough to show that $X_+ \cap X_- =\{id\}$. But $X_+$ is contained in the positive unipotent group and $X_-$ in the negative unipotent group, so their intersection must be trivial. 
	
	As a variety $U_P$ is isomorphic to an affine space - its Lie algebra $\mathfrak{u}_P$. The induced action of $\T$ on the linear space $\mathfrak{u}_P$ is linear - it is \red{a} part of the adjoint representation of $G$.
	It follows that both  $X_+$ and $X_-$ are BB-cells of \red{a} linear action on a linear space and therefore linear subspaces. Thus\red{,} the product $X_+\times X_-$ is isomorphic to \red{an} affine space of dimension equal to dimension of $U_P$.
	Thus\red{,} the map $\theta$ is an algebraic endomorphism of an affine space which is injective on points. The Ax–Grothendieck theorem (cf. \cite[Theorem 10.4.11.]{Ax,Gro})
	implies that it is bijective on points.
	Affine space is smooth and connected so the Zariski main theorem (cf. \cite[Theorem 4.4.3]{EGA3.1})
	 implies that $\theta$ is an algebraic isomorphism.
	 
	 To prove the second property it is enough to show the containment
	 $$\theta\left(X_+ \times (U_P\cap (G/P)_y^+)\right) \subset (G/P)_y^+,$$
	 for any fixed point $y$. Note that
	 $$X_+ \subset U_\sigma^+\subset B_\sigma^+.$$
	 Moreover\red{,} the BB-cell $(G/P)_y^+$  is an orbit of the group $B_\sigma^+$ and the morphism $\theta$ coincides with the action of $B_\sigma^+$. So the desired inclusion holds.  
\end{proof}
\begin{rem}
	One can show alternatively that $X_-$ and $X_+$ are isomorphic to affine spaces using theorem 1.5 of \cite{JeSi}. It is also possible to omit the Ax-Grothendieck theorem by using classical results of the theory of Lie groups. Namely lemma 17 from \cite{Stei} implies that the map $\theta$ is bijective.
\end{rem}

\newcommand{\etalchar}[1]{$^{#1}$}

\end{document}